\documentclass[11pt]{amsart} 
 
\usepackage[dvips]{graphics} 
\usepackage{color,amssymb,amsbsy,amsmath,amsfonts,amssymb,
amscd,epsfig,times,graphics}

\def\Re{{\sf Re}\,}

\newtheorem{theo}{Theorem}[section] 
\newtheorem{prop}[theo]{Proposition}
\newtheorem{ex}[theo]{Example}
\newtheorem{lem}[theo]{Lemma} 
\newtheorem{cor}[theo]{Corollary}
\newtheorem{rem}[theo]{Remark}
\theoremstyle{definition}
\newtheorem{defi}[theo]{Definition}
 
\numberwithin{equation}{section}

\newcommand{\C}{\mathbb C} 
\newcommand{\R}{\mathbb R} 
\newcommand{\N}{\mathbb N} 
\newcommand{\B}{\mathbb B} 
\newcommand{\D}{\mathbb D}

\begin{document} 
\begin{abstract}
We introduce a prime end-type  theory on complete Kobayashi hyperbolic manifolds using horosphere sequences. This allows to introduce a new notion of boundary---new even in the unit disc in the complex space---the horosphere boundary, and a topology on the manifold together with its horosphere boundary, the horosphere topology.  We prove that a bounded strongly pseudoconvex domain endowed with  the horosphere topology is homeomorphic to its Euclidean closure, while for the polydisc such a horosphere topology is not even Hausdorff and is different from the Gromov topology. We use this theory to study boundary behavior of univalent maps from bounded strongly pseudoconvex domains. 
\end{abstract} 

\title{Horosphere topology}

\author[F. Bracci \and H. Gaussier]{Filippo Bracci$^\ast$ and Herv\'e Gaussier}

\address{\begin{tabular}{lll}
 Filippo Bracci & & Herv\'e Gaussier\\
Dipartimento di Matematica & & Univ. Grenoble Alpes, CNRS, IF, F-38000 Grenoble, France\\
Universit\`a di Roma ``Tor Vergata''& & \\
Via della Ricerca Scientifica, 1& &
 \\
00133 Roma& & 
\\
Italy & &
\end{tabular}
}
\email{fbracci@mat.uniroma2.it \ herve.gaussier@univ-grenoble-alpes.fr } 

\thanks{$^\ast$Supported by the ERC grant ``HEVO - Holomorphic Evolution Equations'' n. 277691}
\maketitle 

\tableofcontents

\section{Introduction}

Carath\'eodory prime ends theory is one of the most powerful tools for studying boundary behavior of univalent functions in the unit disc $\D\subset \C$. Given a simply connected domain $\Omega\subset \C$, one can define an abstract boundary $\partial_C \Omega$, the Carath\'eodory boundary, whose points, called prime ends, are given as  the set of equivalent classes of null chains (see, {\sl e.g.}, \cite{CL, Ep, Po} or Section \ref{Cara}). Then one can give a natural topology, the Carath\'eodory topology, to the space $\hat{\Omega}_C:=\Omega\cup \partial_C\Omega$. It turns out that $\hat{\D}_C$ is homeomorphic to the Euclidean closure $\overline{\D}$ and, if $f:\D \to \Omega$ is a biholomorphism, then $f$ extends to a homeomorphism from $\hat{\D}_C$ to $\hat{\Omega}_C$. The link between the Carath\'eodory boundary and the boundary behavior of univalent functions is provided by impressions and principal parts of prime ends. In particular, if the impression of each prime end of $\Omega$ is just one point,  the 
map $f$ extends continuously. 

Carath\'eodory prime ends theory is defined by using Euclidean objects (the null chains, which are sequences of Jordan arcs ending at the frontier of $\Omega$), and the ultimate reason why it works is because univalent mappings in $\D$ are (quasi-)conformal. Indeed, Carath\'eodory's theory can be generalized to domains in $\R^n$ for quasi-conformal mappings (see \cite{Va} and \cite{HK}). Other generalizations of the prime ends theory in metric spaces have been studied in \cite{ABBS}. The problem with those generalizations to higher dimension, is that they cannot be applied to univalent maps without adding some extra hypotheses (in general, a univalent map in higher dimension is not quasi-conformal).  

Different types of boundaries were introduced in other contexts. For instance, the Gromov boundary was introduced by M. Gromov \cite{Gr} in hyperbolic groups with the purpose of a better understanding of the growth of some groups at infinity. The construction is valid in hyperbolic metric spaces, namely metric spaces in which geodesic triangles are thin. In most situations, the space is assumed geodesic, meaning that any two points may be joined by a geodesic, and proper, meaning that closed balls are compact.
We recall that a geodesic ray, in a geodesic metric space $(X,d)$, is an isometry from $[0,+\infty[$ to $X$ such that the length of the segment $\gamma([0,t])$ is equal to $d(\gamma(0),\gamma(t))$ for every $t \geq 0$.
 Two geodesic rays $\gamma_1$ and $\gamma_2$ are equivalent if there is some $c>0$ such that $d(\gamma_1(t), \gamma_2(t)) \leq c$ for every $t \geq 0$.
 
 The Gromov boundary of $X$ is the quotient of the set of geodesic rays whose origin is some fixed base point $x \in X$ by that equivalence relation. Moreover, it does not depend on the chosen base point. See Section 7 for the precise definitions.
 

In this paper we introduce a completely new prime ends theory defined via horospheres related to sequences. Horospheres have been used pretty much in geometric function theory in one and several variables, especially for studying iteration theory, Julia's Lemma, Denjoy-Wolff theorems (see, {\sl e.g.}, \cite{ABZ, Aba, AR, Bud, BKR, RS, Fr} and references therein), and they are a particular instance of a general notion of horospheres  in locally complete metric spaces, see \cite{EO, BGS}.  In complex geometry, horospheres defined by using complex geodesics are sometimes called Busemann horospheres. In strongly convex domains with smooth boundary, thanks to Lempert's theory \cite{L}, horospheres turn out to be level sets of a pluricomplex Poisson kernel \cite{BP, BPT} and in the unit ball $\B^n$ they are just ellipsoids internally tangent to the boundary of the ball at one point. 

Our point of view is different than in the previous works. 
To be precise, let $M$ be a Kobayashi complete hyperbolic manifold, let $K_M$ denote its Kobayashi distance, and let $x\in M$. Given a  compactly divergent sequence  $\{u_n\}$ in $M$, and $R>0$, we define the {\sl horosphere} relative to $\{u_n\}$ with radius $R$ by
\[
E_x(\{u_n\}, R):=\{w\in M: \limsup_{n\to \infty}[K_M(w,u_n)-K_M(x,u_n)]<\frac{1}{2}\log R\}.
\]
The sequence is {\sl admissible} if $E_x(\{u_n\}, R)\neq \emptyset$ for all $R>0$. We introduce an equivalence relation on the set of all admissible sequences by declaring equivalent  two admissible sequences if  every horosphere relative to one sequence is contained in a horosphere of the other and vice versa (see Section \ref{Admiss} for precise statements and definitions). The {\sl horosphere boundary }ÃÂ $\partial_H M$ is the set of equivalence classes of admissible sequences. Let $\hat{M}:=M\cup \partial_H M$. Then, using horospheres, we define a topology on $\hat{M}$ which induces on $M$  its natural topology. We call {\sl horosphere topology} such a topology (see Section \ref{topology}). Since this topology is defined via the Kobayashi distance $K_M$, it turns out that if $F:M\to N$ is a biholomorphism then $F$ extends naturally to a homeomorphism $\hat{F}:\hat{M}\to \hat{N}$. 

The first main result of the paper, is the following generalization of Carath\'eodory's theorem (see Theorem \ref{main-pseudo}):

\begin{theo}\label{main-pseudo-intro}
Let $D\subset \C^N$ be a bounded strongly pseudoconvex domain with $C^3$ boundary. Then $\hat{D}$ endowed with the horosphere topology  is homeomorphic to $\overline{D}$ (closure in $\C^N$) endowed with the Euclidean topology.
\end{theo}

The somewhat unnatural technical assumption on $C^3$ regularity of the boundary of $D$ is needed in order to apply some theory of complex geodesics in strongly convex domains.


As a matter of fact, the horosphere boundary of the polydisc is not Hausdorff; hence, we have another proof of the well known fact that strongly pseudoconvex domains cannot be biholomorphic to polydiscs. More generally, using this result, we can prove that there is no holomorphic isometric embeddings of a polydisc into any strongly pseudoconvex domain.

As in Carath\'eodory's prime ends theory, for domains in $\C^N$ (or more generally in $\mathbb C\mathbb P^N$), we can define {\sl horosphere impressions} and {\sl horosphere principal parts}. In particular, for  a bounded strongly pseudoconvex domain $D\subset \C^N$  with $C^3$ boundary, the horosphere impressions always reduce to one point at the boundary. Therefore, if $F:D\to \Omega$ is a biholomorphism, the limit of $F$ at a point $p\in \partial D$ is given by the impression of $\hat{F}(\underline{x}_p)$, where $\underline{x}_p$ is the point of $\partial_H D$  corresponding to $p$ under the isomorphism of Theorem \ref{main-pseudo-intro}. In particular, $F$ extends continuously on $\overline{D}$ if and only if the impressions of each point of $\partial_H \Omega$ is just one point in $\partial \Omega$---and this also gives another proof of homeomorphic extension of biholomorphisms among strongly pseudoconvex domains. On the other hand, non-tangential limits (in fact, a larger notion of limits which we call 
E-limits) can be controlled using the principal part of horospheres, similarly to what happens in Carath\'eodory's  theory for the principal parts of prime ends (see Section \ref{App}). 

We apply the horosphere theory to study biholomorphisms $F:D\to \Omega$ from a strongly pseudoconvex domain $D$ to a convex domain $\Omega$. If the domain $\Omega$ is strongly convex with smooth boundary, then Fefferman's theorem \cite{Fef} implies that the map $F$ extends as a diffeomorphism from $\overline{D}$ to $\overline{\Omega}$.  In case $\Omega$ has no boundary regularity, some other conditions on the behavior of the Kobayashi distance can be useful to obtain continuous extension (see \cite{Zi2} where A. Zimmer proves the continuous extension of Kobayashi isometric embeddings of a bounded convex domain with $\mathcal C^{1,\alpha}$ boundary into a strictly $\mathbb C$-convex domain with $\mathcal C^{1,\alpha}$ boundary, \cite{BZ} and references therein). However, if nothing is assumed on $\Omega$, very little is known, even when $D=\B^n$,  the unit ball. In fact, a conjecture of Muir and Suffridge  \cite{MS, MS2} states that if $F:\B^n\to \C^N$ is univalent and its image is  convex, then  $F$  extends continuously to $\partial\B^n$ except at at most two infinite singularities.  

Using in an essential way the theory developed in Section \ref{convexsec} and the theory of Gromov hyperbolicity, not only we give an affirmative answer to the Muir and Suffridge conjecture in case of {\sl bounded} convex domains (with no boundary regularity assumed), but also we prove homeomorphic extension without any additional assumption. The  Muir and Suffridge conjecture has been then completely settled by the  two authors in \cite{BG}, using material from this paper. The result we prove here is the following  (see Corollary \ref{continuo-ball}):

\begin{theo}
Let $D\subset \C^N$ be a bounded strongly convex domain with $C^3$ boundary, for instance, $D=\B^n$. Let $F:D \to \Omega$ be a biholomorphism. If $\Omega$ is a bounded convex domain,  then $F$ extends as a homeomorphism  from $\overline{D}$ to $\overline{\Omega}$. 
\end{theo}
The same result holds for univalent mappings from bounded strongly pseudoconvex domains whose ima\-ge is bounded and {\sl strictly $\C$-linearly convex}, see Section \ref{App}. 

Our approach relies in an essential way on the horosphere topology and on the hyperbolicity theory; it gives an example of a deep interaction between metric properties and topological properties of a complex manifold.


As a spin off result of our work, we prove a Denjoy-Wolff Theorem for bounded convex domains $D\subset \C^N$ biholomorphic   to strongly convex domains and for strictly $\C$-linearly convex domains biholomorphic to bounded strongly pseudoconvex domains. Namely, we prove that if $f$ is a holomorphic self-map of $D$ without fixed points then the sequence of its iterates converges to exactly one boundary  point (see Proposition \ref{Denjoy}). 

We also compare the horosphere boundary we introduced with the Gromov boundary. For strongly pseudoconvex domains, by \cite[Theorem 1.4]{BB}, the Gromov boundary is homeomorphic to the Euclidean boundary, thus, by Theorem \ref{main-pseudo-intro}, homeomorphic to the horosphere boundary. Therefore, we examine in details the interesting case of the bidisc. It is standard that the bidisc is not Gromov hyperbolic. We prove that the topology of the Gromov boundary of the bidisc is not trivial while the horosphere topology of the bidisc is trivial, thus the two boundaries are not homeomorphic. 

The plan of the paper is the following. In Section \ref{Cara}, we consider the case of simply connected domains in $\C$, both to explain  the ideas underlying the horosphere theory in some simple case like the unit disc and to compare horosphere topology with  the Carath\'eodory topology. In Section \ref{Admiss} we introduce the notion of admissible sequences and define the horosphere boundary. In Section \ref{topology}, we introduce the horosphere topology for a general complete hyperbolic complex manifold, and the notions of impressions and principal parts for domains in $\C^N$. In Section \ref{strongly}, we turn our attention to the case of strongly pseudoconvex domains, we relate horosphere sequences with Abate's big and small horospheres, and we prove Theorem \ref{main-pseudo-intro}. Next, in Section \ref{convexsec} we concentrate on convex domains. After proving some preliminary results on horospheres for  hyperbolic convex domains, we consider  
convex domains biholomorphic to strongly pseudoconvex domains and then  bounded convex domains biholomorphic to strongly convex domains.  In Section \ref{Gromov}, we face the natural question of comparing horosphere boundaries and Gromov's different notions of boundaries.  Finally, in Section \ref{App}, we prove our extension results using the theory we developed in Section \ref{convexsec} and Gromov's theory of hyperbolic metric spaces.

\vskip 0,2cm
{\sl Acknowledgements. The authors wish to thank Andrew Zimmer for fruitful conversations. The authors also thank the referee for several useful comments which improved the original manuscript.}

\section{Carath\'eodory prime ends theory vs. Horosphere topology in the unit disc}\label{Cara}

In this section we introduce the horosphere topology for simply connected domains in $\C$ and compare this theory with the classical theory of Carath\'eodory. The section does not contain any material which will be used later on, does not contain any proof of stated facts, and can be harmlessly skipped. However,  the horosphere theory we introduce is new also in dimension one. Moreover,  looking first at an easy case like the unit disc might simplify comprehension of the several complex variables case. Therefore, we decided to add this section.

We start by briefly recalling Carath\'eodory's prime ends theory (we refer to \cite{Ep, CL, Po} for details). Let $D\subset \C$ be a simply connected domain. A {\sl cross-cut} is a Jordan arc or a Jordan curve in $\overline{D}$ such that its interior belongs to $D$ and whose two end points belong to $\partial D$ (if the domain $D$ is unbounded, we consider its closure in the Riemann sphere $\mathbb C\mathbb P^1$). Every cross-cut $C$ of $D$ divides $D$ into two connected components by Jordan's theorem, and we say that $C$ separates two sets $A, B\subset D$ if $A$ belongs to one connected component of $D\setminus C$ and $B$ to the other. A {\sl null chain} $(C_n)$ is a sequence of cross-cuts such that $C_j\cap C_k=\emptyset$ for $j\neq k$,  each $C_n$ separates $C_{n+1}\cap D$ from $C_0\cap D$ in $D$ for all $n\geq 1$ and $\hbox{diam}(C_n)\to 0$, where $\hbox{diam}(C_n)$ denotes the diameter in the Euclidean metric or in the spherical metric if the domain is unbounded.  For $n\geq 1$ we denote by $V_n$ the {\sl interior part} of $C_n$, that is, the connected component of $D\setminus C_n$ which does not contain $C_0\cap D$. 

We say that two null chains $(C_n)$ and $(C_n')$ are equivalent  if there exists $n_0\in \N$ such that for every $n>n_0$, $n\in \N$ there exists $m\in \N$ such that $V_m\subset V'_{n}$ and $V'_{m}\subset  V_{n}$ (here $V_n$ is the interior part of $C_n$ and $V_n'$ is the interior part of $C_n'$). An equivalence class of null chains is called a {\sl prime end}. The set of all prime ends is denoted by $\partial_C D$ and it is called the {\sl Carath\'eodory boundary} of $D$. Let $\hat{D}_C:=D\cup \partial_C D$.

We give a topology on $\hat{D}_C$ as follows. A sequence $\{z_n\}\subset D$ converges to a prime end $\underline{x}^C\in \partial_C D$ if there exists a null chain $(C_n)$ representing $\underline{x}^C$ such that for every $N\in \N$ the sequence $\{z_n\}$ is eventually contained in $V_N$. A sequence $\{\underline{x}^C_m\}\subset \partial_C D$ converges to $\underline{x}^C\in \partial_C D$ if there exist null chains $(C_n^m)$ representing $\underline{x}^C_m$ and a null chain $(C_n)$ representing $\underline{x}^C$ such that for every $N\in \N$ there exists $m_0\in \N$ such that for each $m\geq m_0$ the sequence $(C_n^m)$ is eventually contained in $V_N$. On $D$ we keep the Euclidean topology. The topology generated by the previous definitions is called the {\sl Carath\'eodory topology}. 
Two main results of Carath\'eodory theory are the following:
\begin{itemize}
\item $\hat{\D}_C$ endowed with the Carath\'eodory topology is homeomorphic to the closed unit disc $\overline{\D}$ endowed with the Euclidean topology.
\item If $D_1, D_2$ are two simply connected domains, $f:D_1 \to D_2$  a biholomorphism, then $f$ extends to a homeomorphism $\hat{f}_C: \widehat{D_1}_C\to \widehat{D_2}_C$. 
\end{itemize}
Given a prime end $\underline{x}^C\in \partial_C D$, the {\sl prime end impression} is defined by 
\[
\hbox{I}^C_D(\underline{x}^C):=\bigcap_{n\geq 0}ÃÂ \overline{V_n},
\]
where $\{V_n\}$ is the interior part of any null chain $(C_n)$ representing $\underline{x}$ and the closure has to be understood in $\mathbb C\mathbb P^1$ in case of unbounded domains. Note that, $p\in \hbox{I}^C_D(\underline{x}^C)$ if and only if there exists a sequence $\{z_n\}\subset D$ such that $z_n\to p$ in the Euclidean topology and $z_n\to \underline{x}^C$ in the Carath\'eodory topology.

Let $f:\D \to D\subset \C$ be a biholomorphism, $\zeta\in \partial \D$. We denote by $\Gamma(f;\zeta)$ the cluster set of $f$ at $\zeta$, namely, $p\in \Gamma(f;\zeta)$ if there exists a sequence $\{z_n\}\subset \D$ converging to $\zeta$ such that $f(z_n)\to p$. Now, the point $\zeta$  corresponds to a point $\underline{x}^C_\zeta\in \partial_C \D$. A null chain representing $\underline{x}^C_\zeta$ is given by $(C_n)$ with $C_n:=\{z\in \overline{\D}: |z-\zeta|=r_n\}$, where $\{r_n\}$ is any strictly decreasing sequence of positive numbers converging to $0$. Therefore, if $\{z_n\}\subset \D$ converges to $\zeta$ in the Euclidean topology, then it also converges to  $\underline{x}^C_\zeta$ in the Carath\'eodory topology. Now,  one can choose the sequence $\{r_n\}$ in such a way that $(f(C_n))$ is  a null chain  in $D$ which represents $\hat{f}_C(\underline{x}^C_\zeta)$. The sequence $\{f(z_n)\}$ belongs eventually to the interior part of each $\hat{C}_n$, hence, it accumulates to points in the prime ends 
impression of $\hat{f}_
C(\underline{x}^C_\zeta)$. Namely,
\[
\Gamma(f;\zeta)=\hbox{I}^C_D(\hat{f}_C(\underline{x}^C_\zeta)).
\]
In particular, a biholomorphism $f:\D \to D$ extends continuously to $\partial \D$ if and only if the  impression of each prime end is one point. 

The {\sl principal part }ÃÂ $\hbox{II}^C_D(\underline{x}^C)$ of a prime end $\underline{x}^C\in \partial_C D$, is defined as follows. A point $p\in \partial D$ belongs to $\hbox{II}^C_D(\underline{x}^C)$ if for every open neighborhood $U$ of $p$ there exists a null chain $(C_n)$ representing $\underline{x}^C$ such that $C_n\subset U$ for all $n\in \N$. Given a biholomorphism $f:\D \to D$ and $\zeta\in \partial \D$, let denote by $\Gamma_{NT}(f;\zeta)$ the cluster set of $f$ along non-tangential sequences. Namely, $p\in \partial \D$ belongs to $\Gamma_{NT}(f;\zeta)$ if there exists a sequence $\{z_n\}\subset \D$, converging non-tangentially at $\zeta$ such that $f(z_n)\to p$. Then it holds
\[
\Gamma_{NT}(f;\zeta)=\hbox{II}^C_D(\hat{f}_C(\underline{x}^C_\zeta)).
\]

Null chains are Euclidean objects, and the fact that a biholomorphism maps null chains almost into  null chains, allowing to extend the map as  a homeomorphism on the Carath\'eodory boundary, relies strongly on quasi-conformality of biholomorphisms. On the other hand, once this is done, and the homeomorphism between $\hat{\D}_C$ and $\overline{\D}$ is proved, the relation between impressions and unrestricted limits  comes almost for free. 

Here we take a dual point of view. We define a boundary and a topology via the intrinsic hyperbolic distance, in such a way that the homeomorphic extension of biholomorphisms to the newly defined boundary comes for free, but on the other hand, the price we have to pay is that horospheres impressions are less immediate to understand than prime ends impressions in Carath\'eodory theory. 

In order to give some lights on the construction we present in the next sections, we describe here the horosphere topology for the unit disc and simply connected domains in $\C$. A {\sl horosphere} of vertex $\zeta\in \partial \D$ and radius $R>0$ in $\D$ is given by 
\[
E(\zeta, R):=\{z\in \D: \frac{|\zeta-z|^2}{1-|z|^2}<R\}.
\]
It is a disc of radius $R/(R+1)$ contained in $\D$ and tangent to $\partial \D$ at $\zeta$. One can easily show that
\begin{equation}\label{exp-pio}
E(\zeta,R)=\{z\in \D: \lim_{w\to \zeta}[K_\D(z,w)-K_D(0,w)]<\frac{1}{2}\log R\},
\end{equation}
where $K_\D$ denotes the Poincar\'e distance in $\D$. Equation \eqref{exp-pio} is however not yet suitable for being considered in other simply connected domains, essentially because it is related to the point $\zeta$ in the Euclidean boundary of $\D$, which, from an intrinsic point of view of Poincar\'e distance, does not exist. Therefore, instead of considering limits to a given boundary point, we consider sequences $\{z_n\}\subset \D$ such that $\liminf_{n\to \infty}K_\D(z_n,0)=\infty$ (namely, we consider sequences which, from an Euclidean point of view, go to the boundary). For $R>0$ we define 
\begin{equation}\label{Edisc}
E^\D(\{u_n\}, R):=\{z\in \D: \limsup_{n\to \infty}[K_\D(z,u_n)-K_\D(0,u_n)]<\frac{1}{2}\log R\}.
\end{equation}
By \eqref{exp-pio}, if the cluster set of $\{u_n\}$ is more than one point, there exists $R>0$ such that $E(\{u_n\}, R)=\emptyset$, while, if $\{u_n\}$ converges to $\zeta\in \partial \D$, then 
\begin{equation}\label{E-E}
E^\D(\{u_n\}, R)=E(\zeta, R)
\end{equation}
 for all $R>0$. Therefore, morally, we replace $\zeta\in \partial \D$ with sequences $\{u_n\}$ such that  $E^\D(\{u_n\}, R)=E(\zeta, R)$ for all $R>0$. To be more formal, we say that a compactly divergent sequence $\{u_n\}\subset \D$ is {\sl admissible} provided $E^\D(\{u_n\}, R)\neq\emptyset$ for all $R>0$. 

If $D\subset\C$ is a simply connected domain, we can define admissible sequences in $D$ using the same token: a sequence $\{u_n\}\subset D$ is admissible if $\liminf_{n\to \infty}K_D(x,u_n)=\infty$ and 
\[
E_x^D(\{u_n\}, R):=\{z\in \D: \limsup_{n\to \infty}[K_D(z,w)-K_D(0,w)]<\frac{1}{2}\log R\}
\]
is not empty for all $R>0$. Here $K_D$ is the Poincar\'e distance on $D$ and $x\in D$ is a fixed point, whose choice does not play any substantial role. 

Then we define an equivalence relation on admissible sequences by declaring $\{u_n\}$  equivalent to $\{v_n\}$ if for every $R>0$ there exist $R', R''>0$ such that 
\[
E_x^D(\{u_n\}, R')\subset E_x^D(\{v_n\}, R)\subset E_x^D(\{u_n\}, R'').
\]
In case of the unit disc, two admissible sequences $\{u_n\}$ and $\{v_n\}$ are then equivalent if and only if they converge to the same boundary point. 

We denote by $\partial_H D$ the set of all equivalence classes of admissible sequences in $D$  and we call it the {\sl horosphere boundary} of $D$. 
Let $\hat{D}:=D\cup \partial_H D$. We want to give a topology on $\hat{D}$ in such a way that on $D$ it coincides with the Euclidean topology and $\hat{\D}$ is homeomorphic to $\overline{\D}$. 

We start with the last requirement. As we said, if $\underline{x}\in \partial_H \D$, every admissible sequence $\{u_n\}$ representing $\underline{x}$ converges to the same point $\zeta_{\underline{x}}\in \partial \D$. Thus, it is natural to associate  $\underline{x}$ to the point $\zeta_{\underline{x}}$. Moreover, by \eqref{E-E}, we can also think of $\underline{x}$ as the family of horospheres $E(\zeta_{\underline{x}}, R)$, $R>0$.  Now, it is easy to see that  a sequence of points $\{\zeta_j\}\in \partial \D$ converges to $\zeta\in \partial \D$ if and only if for every $R>0$ there exists $n_R\in \N$ such that for all $n\geq n_R$ it holds $E(\zeta_j, R)\cap E(\zeta, R)\neq \emptyset$. This is exactly the definition we can exploit in the general case: let $D\subset \C$ be a simply connected domain. A sequence $\{\underline{x}_j\}\subset \partial_H D$ converges to $\underline{x}\in \partial_H D$ if there exist
admissible sequences $\{u_n^j\}$ representing $\underline{x}_j$ for every $j$, and an admissible sequence $\{u_n\}$ representing $\underline{x}$ such that for every $R>0$ there exists $m_R\in \N$ such that $E^D_x(\{u_n^j\}, R)\cap E^D_x(\{u_n\}, R)\neq \emptyset$ for all $j\geq m_R$. 

Now, we consider convergence from inside $\D$ to the horosphere boundary. It is easy to see that a sequence $\{z_n\}\subset \D$ converges to $\zeta\in \partial \D$ if there exists a sequence of points $\{\zeta_n\}\subset \partial \D$  (not necessarily all different from each other) such that for every $R>0$ there exists $m_R>0$ such that $E(\zeta, R)\cap E(\zeta_j, R)\neq \emptyset$ and $z_j\in E(\zeta_j, R)$  for all $j\geq m_R$. Thus, we can transform this observation into the general definition: a sequence $\{z_j\}\subset D$ converges to $\underline{x}\in \partial_H D$ if there exist
admissible sequences $\{u_n^j\}\subset D$  and an admissible sequence $\{u_n\}\subset D$ representing $\underline{x}$ such that for every $R>0$ there exists $m_R\in \N$ such that $E^D_x(\{u_n^j\}, R)\cap E^D_x(\{u_n\}, R)\neq \emptyset$ and $z_j\in E^D_x(\{u_n^j\}, R)$ for all $j\geq m_R$. 

With the previous definitions, we gave a meaning to the notion of  convergence in $\hat{D}$. In this way we can define the {\sl horosphere topology} on $\hat{D}$ by declaring a subset $C\subset \hat{D}$ to be closed if, whenever $\{z_n\}\subset C$ converges in the sense described above to $z\in \hat{D}$, then $z\in C$. 

By the above discussion, it is clear that $\hat{\D}$ is homeomorphic to $\overline{\D}$. Moreover,  let $f:\D \to D\subset \C$ be a biholomorphism. Since the map $f$ is an isometry between $K_\D$ and $K_D$ and all the constructions are made in terms of hyperbolic distance, it follows at once that $f$  extends to a homeomorphism $\hat{f}: \hat{\D} \to \hat{D}$. 

Being completely intrinsic, the definition of horosphere boundary and horosphere topology is suitable to be generalized  to any complete hyperbolic complex manifold. However, for the same reason, it is harder to relate it to boundary limits, but it can be done. 

Let $D\subset \C$ be a simply connected domain. Given $\underline{x}\in \partial_H D$, we say that $p\in \hbox{I}^H_D(\underline{x})$ if there exists a sequence $\{z_n\}\subset D$ such that $z_n\to p$ in the Euclidean topology and $z_n\to \underline{x}$ in the horosphere topology. We call $\hbox{I}^H_D(\underline{x})$ the {\sl horosphere impression} of $\underline{x}$. 

If $f:\D \to D$ is a biholomorphism, and $\zeta\in \partial \D$, let $\underline{x}_\zeta\in \partial_H\D$ be the point given by the homeomorphism between $\overline{\D}$ and $\hat{\D}$. If $\{z_n\}$ is a sequence in $\D$ converging to $\underline{x}_\zeta$ in the horosphere topology of $\hat{\D}$, then the sequence $\{f(z_n)\}$ converges to $\hat{f}(\underline{x}_\zeta)$ in the horosphere topology of $\hat{D}$. Hence,
\[
\Gamma(f; \zeta)=\hbox{I}^H_D(\hat{f}(\underline{x}_\zeta)).
\]
This implies in particular that $\hbox{I}^H_D(\hat{f}(\underline{x}_\zeta))=\hbox{I}^C_D(\hat{f}_C(\underline{x}^C_\zeta))$. 

We can also define a {\sl horosphere principal part}. Let $\underline{x}\in \partial_H D$. We say that a sequence {\sl E-converges} to $\underline{x}$ if there exists an admissible sequence $\{u_n\}\subset D$ representing $\underline{x}$ such that $\{z_n\}$ is eventually contained in $E^D_x(\{u_n\}, R)$ for every $R>0$. Clearly the definition does not depend on the admissible sequence $\{u_n\}$ chosen. Also, it is evident that if $\{z_n\}$ is E-converging to $\underline{x}$ then it is also converging to $\underline{x}$ in the horosphere topology. For instance, all sequences in $\D$ which converge non-tangentially to one point $\zeta\in \partial \D$, are also E-converging  to $\underline{x}_\zeta$. However, there exist E-convergent sequences in $\D$ which do not converge non-tangentially to a boundary point. Then we let the {\sl horosphere principal part} $\hbox{II}^H_D(\underline{x})$ of a point $\underline{x}\in \partial_H D$ be the set of points $p\in \partial D$ such that there exists a sequence $\{z_n\}\subset D$ converging to $p$ and E-converging to $\underline{x}$. It is clear that $\hbox{II}^H_D(\underline{x})=\bigcap_{R>0}\overline{E^D_x(\{u_n\}, R)}$ where $\{u_n\}$ is any admissible sequence representing $\underline{x}$. 

Now, if $f:\D \to D$ is a biholomorphism, and $\zeta\in\partial \D$, let $\Gamma_E(f;\zeta)$ be the set of points $p\in\partial D$ such that there exists a sequence $\{z_n\}$ E-convergent to $\underline{x}_\zeta$ such that $f(z_n)\to p$. Since all the notions are defined via intrinsic distance, it follows at once that
\[
\Gamma_E(f;\zeta)=\hbox{II}^H_D(\hat{f}(\underline{x}_\zeta))=\bigcap_{R>0}\overline{E^D_x(\{f(u_n)\}, R)}.
\]
Note that $\Gamma_{NT}(f;\zeta)\subseteq \Gamma_E(f;\zeta)$, hence, 
\[
\hbox{II}^C_D(\hat{f}_C(\underline{x}^C_\zeta))\subseteq \hbox{II}^H_D(\hat{f}(\underline{x}_\zeta)).
\]
In  \cite{GP}, Gaier and Pommerenke constructed examples of univalent mappings $f:\D \to \C$ for which, in our notation,  $\hbox{II}^C_D(\hat{f}_C(\underline{x}^C_\zeta))\neq\hbox{II}^H_D(\hat{f}(\underline{x}_\zeta))$. On the other hand, Twomey \cite{Tw} proved that if $f:\D\to \C$ is starlike then $\Gamma_E(f;\zeta)$ consists of one point for every $\zeta\in \partial \D$. Therefore, in this case, $\hbox{II}^C_D(\hat{f}_C(\underline{x}^C_\zeta))=\hbox{II}^H_D(\hat{f}(\underline{x}_\zeta))$. 

\section{Admissible  sequences, Busemann sequences and horosphere boundary}\label{Admiss}

Let $M$ be a (connected) complex manifold of complex dimension $N$.  We denote by $K_M$ the Kobayashi distance on $M$. We assume that $M$ is (Kobayashi) complete hyperbolic, namely that the metric space $(M,K_M)$ is complete.

Let $\{u_n\}$ be a sequence of points in $M$. For $R>0$ and $x\in M$, we denote by $E_x(\{u_n\},R)$ the {\sl (sequence) horosphere} defined by
$$
E_x(\{u_n\},R):=\{w \in M | \limsup_{n \rightarrow \infty}\left(K_M(w,u_n)-K_M(x,u_n)\right) < \frac{1}{2}\log R\}.
$$

\begin{lem}\label{changebasept}
Let $x, y\in M$. Let $\{u_n\}$ be a sequence in $M$. Then there exist $\alpha, \beta>0$ such that for all $R>0$
\[
E_y(\{u_n\}, \alpha R)\subset E_x(\{u_n\}, R)\subset E_y(\{u_n\}, \beta R).
\]
\end{lem}
\begin{proof}
Let $w\in E_x(\{u_n\},R)$. Let $\frac{1}{2}\log \beta:=\limsup_{n\to \infty}[K_M(x,u_n)-K_M(y,u_n)]$. Note that $\beta\in (0,\infty)$ since 
\[
-\infty<-K_M(x,y)\leq \frac{1}{2}\log \beta\leq K_M(x,y)<+\infty.
\]
 Then
\[
\begin{split}
\limsup_{n\to \infty}[K_M(w,u_n)-K_M(y,u_n)]&\leq \limsup_{n\to \infty}[K_M(w,u_n)-K_M(x,u_n)]\\&+\limsup_{n\to \infty}[K_M(x,u_n)-K_M(y,u_n)]<\frac{1}{2}\log R+\frac{1}{2}\log \beta,
\end{split}
\]
hence $w\in E_y(\{u_n\}, \beta R)$. 

Similarly, if $\frac{1}{2}\log \alpha:=-\limsup_{n\to \infty}[K_M(y,u_n)-K_M(x,u_n)]$, we obtain $E_y(\{u_n\}, \alpha R)\subset E_x(\{u_n\}, R)$.
\end{proof}

\begin{defi}\label{admis-def}
Let $M$ be a complete hyperbolic manifold. Let $x\in M$. A sequence $\{u_n\}$ is {\sl admissible} if

\begin{itemize}
\item[(i)] $\lim\inf_{n \rightarrow \infty}K_M(x,u_n) = \infty$,

\item[(ii)] $\forall R > 0,\ E_x(\{u_n\},R) \neq \emptyset$.
\end{itemize}
We denote by $\Lambda_{M}$ the set of admissible sequences.
\end{defi}

By Lemma \ref{changebasept}, the definition of admissible sequence does not depend on the base point $x$ chosen to define horospheres.

\vskip 0,1cm
The construction of admissible sequences is valid for any geodesic, proper, complete metric space. The existence of admissible sequences in that general situation is given by the following proposition, communicated to the authors by Andrew Zimmer. In case $M$ is a complete hyperbolic manifold for the Kobayashi metric, the existence of a geodesic joining any two points comes from the Hopf-Rinow Theorem in locally compact, complete length spaces. We state Proposition~\ref{adm-prop} for complete Kobayashi hyperbolic manifolds, to keep the context of the paper.

\begin{prop}\label{adm-prop}
 Let $M$ be a complete hyperbolic manifold and let $x \in M$. Then every sequence $\{u_n\}$ of points in $M$, such that $\lim_{n \rightarrow \infty}K_M(u_n,x) = + \infty$, admits an admissible subsequence. In particular, $\Lambda_M\neq \emptyset$.
\end{prop}

\begin{proof}
 Consider any sequence $\{u_n\}$ converging to infinity for the Kobayashi distance.
The function $b_n :  M \ni z\mapsto K_M(u_n,z) - K_M(u_n,x)$ is 1-Lipschitz for every $n$. According to the Ascoli-Arzel\`a Theorem, the sequence $\{b_n\}$ admits a subsequence $\{b_{\varphi(n)}\}$ that converges, uniformly on compact subsets of $M$, to some function $b$.
For every $n \geq 0$, let $\gamma_{\varphi(n)} : [0,T_n] \to M$ be a real geodesic joining $x$ to $u_{\varphi(n)}$.
Then by the definition of $b_n$ we get:
$$
\begin{array}{lll}
- \forall 0 \leq t \leq T_{\varphi(n)}, b_{\varphi(n)}(\gamma_{\varphi(n)}(t)) & = & K_D(\gamma_{\varphi(n)}(T_{\varphi(n)}),\gamma_{\varphi(n)}(t)) - K_D(\gamma_{\varphi(n)}(T_{\varphi(n)}),\gamma_{\varphi(n)}(0)) \\
& = & -K_D(\gamma_{\varphi(n)}(t),\gamma_{\varphi(n)}(0))\\
& = & -t.
\end{array}
$$
Since $(M,K_M)$ is a proper metric space, it follows from the Ascoli-Arzel\`a Theorem that the sequence $\{\gamma_{\varphi(n)}\}$ admits a subsequence $\{\gamma_{\sigma(n)}\}$ that converges, locally uniformly, to some $\gamma$. In particular we get $b(\gamma(t)) = -t$ for every $t \geq 0$.

Let us fix $r>0$ and consider a real number $t_r$ satisfying $t_r > -\frac{1}{2}\log(r)$. Then for sufficiently large $n$ we have:
$$
b_{\sigma(n)}(\gamma(t_r)) < \frac{1}{2}\log(r),
$$
meaning that $\gamma(t_r) \in E_x(\{u_{\sigma(n)}\},r)$.
\end{proof}

\vskip 0,5cm
We collect here some properties of horospheres:

\begin{prop}\label{propertyhoro} Let $M$ be a complete hyperbolic manifold and let $x\in M$. Let $\{u_n\}$ be an admissible sequence. Then
\begin{enumerate}
\item $E_x(\{u_n\},R)$ is open for every $R>0$. 
\item If $0<R<R'$ then $E_x(\{u_n\},R)\subset E_x(\{u_n\},R')$. 
\item $\bigcap_{R>0} E_x(\{u_n\}, R)=\emptyset$.
\end{enumerate}
\end{prop}

\begin{proof}
(1) If $w\in E_x(\{u_n\},R)$, then there exists $0<R'<R$ such that
\[
\limsup_{n \rightarrow \infty}\left(K_M(w,u_n)-K_M(x,u_n)\right)=\frac{1}{2}\log R'.
\]
Hence, if $z\in B_K(w,\delta)$, the Kobayashi ball of center $w$ and radius $\delta=\frac{1}{2}\log \frac{R}{R'}$, it follows
\begin{equation*}
\begin{split}
K_M(z,u_n)-K_M(x,u_n) &= K_M(z,u_n)-K_M(w,u_n) + K_M(w,u_n)-K_M(x,u_n)\\ &\leq K_M(z,w)+K_M(w,u_n)-K_M(x,u_n),
\end{split}
\end{equation*}
and then
\[
\limsup_{n\to \infty} [K_M(z,u_n)-K_M(x,u_n)]<\frac{1}{2}\log\frac{R}{R'}+\frac{1}{2}\log R'=\frac{1}{2}\log R.
\]
Therefore, $B_K(w,\delta)\subset E_x(\{u_n\},R)$, and the horosphere is open.

(2) is obvious.

(3) Since for every $w\in M$,
\[
-K_M(x, w)\leq K_M(w,u_n)-K_M(x,u_n),
\]
it follows that if $w\in \bigcap_{R>0} E_x(\{u_n\}, R)$ then $K_M(x, w)=\infty$, a contradiction.
\end{proof}

A less obvious property is the following property which states somewhat that horospheres go ``uniformly to the boundary'':

\begin{prop}\label{outsideK}
Let $M$ be a complete hyperbolic manifold and let $x\in M$. Then for every compact set $K\subset M$ there exists $R_0>0$ such that 
\[
E_x(\{u_n\}, R)\cap K=\emptyset
\]
for all $R\leq R_0$ and all admissible sequences $\{u_n\}\subset M$.
\end{prop}
\begin{proof}
Assume by contradiction this is not the case. Then for every $N\in \N$ there exist $z_N\in K$ and an admissible sequence $\{u_n^N\}\subset M$ such that $z_N\in E_x(\{u_n^N\}, 1/N)$. Up to subsequences, we can assume that $\{z_N\}$ converges to $z_0\in K$ and that for all $N\in \N$ it holds $K_D(z_N, z_0)<1$. Then, using the triangle inequality, for all $n\in \N$ we have
\begin{equation*}
\begin{split}
- K_D(z_0, x)&\leq \limsup_{n\to \infty}[K_D(z_0, u_n^N)-K_D(x, u_n^N)]\\
&\leq \limsup_{n\to \infty}[K_D(z_N,z_0)+K_D(z_N, u_n^N)-K_D(x, u_n^N)]\\
&\leq 1+\limsup_{n\to \infty}[K_D(z_N, u_n^N)-K_D(x, u_n^N)]<1+\frac{1}{2}\log \frac{1}{N},
\end{split}
\end{equation*} 
a contradiction.
\end{proof}

Another natural class of sequences is the following:

\begin{defi}
Let $M$ be a complete hyperbolic manifold, $x\in M$. A sequence $\{u_n\}\subset M$ is called a {\sl Busemann sequence} if 
\begin{itemize}
\item[(i)] $\lim\inf_{n \rightarrow \infty}K_M(x,u_n) = \infty$,

\item[(ii)] for all $z\in M$ the limit $\lim_{n\to \infty}[K_M(z, u_n)-K_M(x, u_n)]$ exists. 
\end{itemize}
We denote by $\mathcal B_{M}$ the set of Busemann sequences.
\end{defi}

It is not difficult to show that the definition of Busemann sequence does not depend on the base point $x\in M$ chosen.

Busemann sequences have been introduced (under the name {\sl horosphere sequences}) in balanced bounded convex domains in \cite{KKR2}, and later studied in bounded convex domains in \cite{Bud}. The proof of the following result is an adaptation of  the ideas contained in those papers.

\begin{prop}\label{Bus}
Let $M$ be a complete hyperbolic manifold, $x\in M$. 
\begin{enumerate}
\item If $\{u_n\}\subset M$ satisfies $\lim_{n\to \infty}K_M(x, u_n)=\infty$, then there exists a subsequence $\{u_{n_k}\}$ of $\{u_n\}$ which is a Busemann sequence. Therefore, $\mathcal B_M\neq \emptyset$.
\item If $\{u_n\}\subset M$ is an admissible sequence, then there exists a subsequence $\{u_{n_k}\}$ of $\{u_n\}$ which is a Busemann admissible sequence and $E_x(\{u_n\}, R)\subseteq E_x(\{u_{n_k}\}, R)$ for all $R>0$.
\item If for every two points $z, w\in M$, $z\neq w$ there exists a complex geodesic ({\sl i.e.}, an analytic disc which is an isometry between $K_\D$ and $K_D$) which contains $z,w$ then every Busemann sequence is an admissible sequence. 
\end{enumerate}
\end{prop}

\begin{proof}
(1) Let $\{z_m\}_{m\in \N}$ be a dense set in $M$. Arguing as in the proof of Proposition \ref{adm-prop}, one can find a subsequence $\{u_{n_j}\}$ such that $f_j(z):=K_D(z, u_{n_j})-K_D(x, u_{n_j})$ has the property that $\lim_{j\to \infty} f_j(z_m)$ exists for all $m\in \N$. If $z\in M$ and $\{z_{m_k}\}$ is a subsequence converging to $z$, by the triangle inequality it follows 
\[
|f_j(z)-f_j(z_{m_k})|\leq K_D(z, z_{m_k}).
\]
From this it is easy to see that $\lim_{j\to \infty}f_j(z)$ exists, hence $\{u_{n_j}\}\in \mathcal B_M$.

(2) If $\{u_n\}\in \Lambda_M$, then by (1) we can extract a Busemann subsequence $\{u_{n_k}\}$. Let $R>0$ and let $z\in E_x(\{u_n\}, R)$. Then
\[
\limsup_{k\to \infty}[K_D(z,u_{n_k})-K_D(x, u_{n_k})]\leq \limsup_{n\to \infty}[K_D(z,u_{n})-K_D(x, u_{n})]<\frac{1}{2}\log R,
\]
which proves that $z\in E_x(\{u_{n_k}\}, R)$.

(3) Assume that every two points of $M$ belong to a complex geodesic and let $\{u_n\}\in \mathcal B_M$. Fix $R>0$ and let $0<r<\min\{1,R\}$. Let $B_K(x, -\frac{1}{2}\log r)$ be the Kobayashi ball of center $x$ and radius $-\frac{1}{2}\log r>0$. Since $K_D(x, \{u_n\})\to \infty$, we can assume that $u_n\not\in B_K(x, -\frac{1}{2}\log r)$ for all $n\in \N$. By hypothesis $x$ and $u_n$ are contained in a complex geodesic.  Hence, there exists $q_n$ in the same complex geodesic such that $K_D(x, q_n)=-\frac{1}{2}\log r$ and
\[
K_D(x, u_n)=K_D(x, q_n)+K_D(q_n, u_n)=-\frac{1}{2}\log r+K_D(q_n, u_n).
\]
Note that $q_n\in \partial B_K(x, -\frac{1}{2}\log r)$ and the latter is compact. Hence we can extract a converging subsequence $\{q_{n_k}\}$ converging to some $q\in \partial B_K(x, -\frac{1}{2}\log r)$. Now, taking into account that $\{u_n\}$ is a Busemann sequence, we have
\begin{equation*}
\begin{split}
\limsup_{n\to \infty}[K_D(q,u_n)-K_D(x,u_n)]&=\lim_{k\to \infty}[K_D(q,u_{n_k})-K_D(x,u_{n_k})]\\&\leq \frac{1}{2}\log r+\lim_{k\to \infty}K_D(q, q_{n_k})=\frac{1}{2}\log r,
\end{split}
\end{equation*}
that is $q\in E_x(\{u_n\}, R)$. Hence $\{u_n\}$ is admissible.
\end{proof}

\begin{defi}\label{equiv-def} Let $M$ be a complete hyperbolic manifold and let $x\in M$.
\begin{itemize}
\item[(i)] Two sequences $\{u_n\}$, $\{v_n\}$ in $\Lambda_M$ are {\sl equivalent} and we write $\{u_n\} \sim \{v_n\}$ if for every $R>0$ there exist $R'>0$ and $R''>0$ such that 
\[
E_x(\{u_n\}, R') \subseteq E_x(\{v_n\},R), \quad E_x(\{v_n\},  R'')\subset E_x(\{u_n\},R).
\]
\item[(ii)] We denote by $\partial_HM$ the {\sl horosphere boundary of $M$} defined by $\partial_HM:=\Lambda_M/\sim$.
\end{itemize}
\end{defi}

It is easy to see that $\sim$ is an equivalence relation on $\Lambda_M$. Moreover, by Lemma \ref{changebasept}, the definition of equivalent admissible sequences is independent of the base point $x$.

\begin{rem}
By Proposition \ref{adm-prop}, the horosphere boundary of a complete hyperbolic manifold is never empty.
\end{rem}

\section{Horosphere topology, impression and principal part}\label{topology}

Let $M$ be a complete hyperbolic complex manifold, $x\in M$, and let $\partial_H M$ be its horosphere boundary. Let $\hat{M}:=M\cup \partial_H M$. We define a topology on $\hat{M}$, which coincides with the topology of $M$ on $M$, as follows:

\begin{defi}\label{conv-out-out}
A sequence $\{\underline{y_m}\}\subset \partial_H M$ converges to $\underline{y}\in \partial_H M$ if there exist admissible sequences $\{u^m_n\}_{n\in \N}$ with $[\{u^m_n\}]=\underline{y_m}$ and an admissible sequence $\{u_n\}$ with $[\{u_n\}]=\underline{y}$ with the property that for every $R>0$ there exists $m_R\in\N$ such that
\begin{equation}\label{def-horo-converge}
E_x(\{u^m_n\}, R)\cap E_x(\{u_n\}, R)\neq \emptyset \quad \forall m\geq m_R.
\end{equation}
\end{defi}

By Lemma \ref{changebasept}, the previous definition of convergence is independent of the base point $x$.

\begin{defi}\label{conv-int-out}
A sequence $\{z_m\}\subset M$ converges to $\underline{y}\in \partial_H M$ if there exist  $\{u_n\}, \{v^j_n\}_{j\in \N}\subset \Lambda_M$ with $[\{u_n\}]=\underline{y}$ and with the property that for all $R>0$  there exists $m_R\in \N$ such that $z_m\in E_x(\{v^m_n\}, R)$,  and $E_x(\{v_n^m\}, R)\cap E_x(\{u_n\}, R)\neq \emptyset$ for all $m\geq m_R$.
\end{defi}
It is clear from  Lemma \ref{changebasept} that this definition does not depend  on the base point $x$.

\begin{rem}\label{conv-in-like-out}
Note that, if the $\{v_n^j\}$ are admissible sequences as in Definition \ref{conv-int-out}, and $\underline{y}_j:=[\{v_n^j\}],$ then $\{\underline{y}_j\}$ converges to $\underline{y}$ in the sense of Definition \ref{conv-out-out}.
\end{rem}

A sequence $\{z_n\}\subset M\subset \hat{M}$ converges to $z\in M$ if $\lim_{n\to \infty}K_M(z_n, z)=0$. Now we can define a topology on $\hat{M}$ as follows:

\begin{defi}
A subset $C\subset \hat{M}$ is {\sl closed} if for every sequence $\{z_n\}\subset C$ converging to $z\in \hat{M}$, it follows that $z\in C$.
\end{defi}

 We call such a topology the {\sl horosphere topology} of $M$ and we denote it by $\mathcal T_H(\hat{M})$.

\begin{rem}\label{boundaryconv}
By definition, a sequence $\{z_n\}\subset \partial_H M$ can converge only to  points which belong to $\partial_H M$ in the horosphere topology.
\end{rem}

\begin{rem}
A relatively compact sequence $\{z_n\}\subset M$ can not converge to a point $\underline{y}\in \partial_H M$ by Proposition \ref{outsideK}. Therefore the horosphere topology restricted on $M$ coincides with the topology of $M$. \end{rem}

If $M, N$ are two complete hyperbolic complex manifolds and $F: M \to N$ is a biholomorphism, then $F$ induces a bijective map $\hat{F}: \hat{M}\to \hat{N}$ as follows. If $z\in M$ then $\hat{F}(z)=F(z)$. If $\underline{y}\in \partial_H M$, let $\{u_n\}$ be an admissible sequence in $M$ which represents $\underline{y}$. Since $F$ is an isometry between $K_M$ and $K_N$, it follows easily that $\{F(u_n)\}$ is an admissible sequence in $N$ and therefore it represents a point $\hat{F}(\underline{y})\in \partial_H N$. Clearly, since $F$ is an isometry between the Kobayashi distance of $M$ and that of $N$, such a point does not depend on the admissible sequence chosen to represent it. 

Moreover, by the same token, it follows that $\hat{F}^{-1}=\widehat{F^{-1}}$ and  $\hat{F}$ and $\widehat{F^{-1}}$ map closed sets onto closed sets in the horosphere topology. Therefore we have

\begin{theo}\label{extension}
Let $M, N$ be complete hyperbolic complex manifolds. Let $F: M \to N$ be a biholomorphism. Then $F$ extends to a homeomorphism $\hat{F}: \hat{M}\to \hat{N}$.
\end{theo}

In case of domains in $\C^N$ one can relate convergence to the horosphere boundary in the horosphere topology with (Euclidean) convergence to the Euclidean boundary. This can be done in two natural ways using horosphere topology, and gives rise to the notion of horosphere impression and horosphere principal part. These notions, as it will be clear later on, are strictly related to extension of univalent maps from bounded strongly pseudoconvex domains. 

As a matter of notation, if $D\subset \C^N$ is a domain, we denote by $\overline{D}^{\mathbb C \mathbb P^N}$ its closure in $\mathbb C \mathbb P^N$.
 
\begin{defi}\label{impr-def}
Let $D$ be a complete hyperbolic domain in $\mathbb C^n$. Let $\underline{x} \in \partial_H D$. We say that a point $p\in \overline{D}^{\mathbb C \mathbb P^N}$ belongs to the {\sl horosphere impression} of $\underline{x}$, denoted by $\hbox{ I}^H_D(\underline{x})$,  if there exists a sequence $\{w_n\}\subset D$ such that $\{w_n\}$ converges to $\underline{x}$ in the horosphere topology and $\{w_n\}$ converges to $p$ in ${\mathbb C \mathbb P^N}$.
\end{defi}

\begin{rem}
By Proposition \ref{outsideK}, for every $\underline{x} \in \partial_H D$ it holds  $\hbox{ I}^H_D(\underline{x})\cap D=\emptyset$ . 
\end{rem}

In order to define the horosphere principal part, we first give a general definition of convergence in the horosphere topology, which, somehow, replaces the notion of non-tangential limits:

\begin{defi}\label{E-lim}
Let $M$ be a complete hyperbolic manifold, $x\in M$. Let $\underline{x} \in \partial_H M$. We say that a sequence $\{z_n\}\subset M$ is {\sl $E\hbox{-}$converging to $\underline{x}$}, and we write
\[
\hbox{E}-\lim_{n\to \infty} z_n=\underline{x},
\]
if for one---and hence any---admissible sequence $\{u_n\}$ representing $\underline{x}$, the sequence $\{z_n\}$ is eventually contained in $E_x(\{u_n\}, R)$ for all $R>0$.
\end{defi}

Note that, according to the definition of horosphere topology, if $\{z_n\}$ is E-converging to $\underline{x}$ then, in particular, it is  converging to $\underline{x}$ in the horosphere topology.

We are now ready to define the {\sl horosphere principal part}:

\begin{defi}\label{princpart-def}
Let $D$ be a complete hyperbolic domain in $\mathbb C^n$. Let $\underline{x} \in \partial_H D$. We say that a point $p\in \overline{D}^{\mathbb C \mathbb P^N}$ belongs to the {\sl horosphere principal part} of $\underline{x}$, denoted by $\hbox{II}^H_D(\underline{x})$,  if there exists a sequence $\{w_n\}\subset D$ such that $\{w_n\}$ E-converges to $\underline{x}$  and $\{w_n\}$ converges to $p$ in ${\mathbb C \mathbb P^N}$.\end{defi}

Clearly, $\hbox{II}^H_D(\underline{x})\subset \hbox{I}^H_D(\underline{x})$ for all $\underline{x}\in \partial_H D$. By the very definition we have
\begin{lem}\label{II-to-hor}
Let $D$ be a complete hyperbolic domain in $\mathbb C^n$, $x\in D$. Let $\underline{x} \in \partial_H D$. Then
\[
\hbox{II}^H_D(\underline{x})=\bigcap_{R>0} \overline{E_x(\{u_n\}, R)}^{\mathbb C \mathbb P^N},
\]
where $\{u_n\}\subset D$ is any admissible sequence representing $\underline{x}$.
\end{lem}

\section{Strongly pseudoconvex domains}\label{strongly}

Abate (see \cite{ABZ, Aba}) defined the small and big horospheres as follows:
\begin{defi}
Let $D\subset \C^N$ be a domain and let $p\in \partial D$. Let $R>0$ and $x\in D$. The {\sl small horosphere} $E_x(p,R)$ of vertex $p$ and radius $R$ is defined by
\[
E_x(p,R):=\{w\in D: \limsup_{z\to p}[K_D(w,z)-K_D(x,z)]<\frac{1}{2}\log R\}.
\]
The {\sl big horosphere} $F_x(p,R)$ of vertex $p$ and radius $R$ is defined by
\[
F_x(p,R):=\{w\in D: \liminf_{z\to p}[K_D(w,z)-K_D(x,z)]<\frac{1}{2}\log R\}.
\]
\end{defi}

The following result follows from  \cite[Lemma 1.1, Thm. 1.7]{ABZ} and \cite[Thm. 2.6.47]{Aba}

\begin{theo}\label{intersF}
Let $D\subset \C^N$ be a bounded strongly pseudoconvex domain with $C^2$ boundary. Let $p\in \partial D$ and $x\in D$. Then 
\begin{enumerate}
\item for every $R>0$, $E_x(p,R)\subset F_x(p,R)$,
\item for every $R'>R>0$, $E_x(p,R)\subset E_x(p,R'), F_x(p,R)\subset F_x(p,R')$,
\item $\cap_{R>0}  F_x(p,R)=\emptyset$,
\item for every $R>0$, $\overline{F_x(p,R)}\cap \partial D=\{p\}$.
\end{enumerate}
Moreover, if $D$ is a strongly convex domain with $C^3$ boundary, then for every $R>0$ it holds $E_x(p,R)=F_x(p,R)$.
\end{theo}

In what follows, we need also the following boundary estimates for the Kobayashi distance, see, {\sl e.g.}, \cite{FR}, \cite[Thm. 2.3.56]{Aba},  \cite[Thm. 2.3.54]{Aba} and \cite[p. 530-531]{BB}. Note that, in those references,  the estimates are proved for a given point in $\partial D$. However, the compactness of  $\partial D$ easily allows  to get uniform estimates. As a matter of notation, if $D\subset \C^n$ is a domain, we denote by $\delta(z)$ the distance of $z\in D$ from the boundary $\partial D$.

\begin{lem}\label{Kobest}
Let $D\subset \C^N$ be a bounded domain with $C^2$ boundary. Then  there exist  $C>0$ and $\varepsilon_0>0$ such that for every $p\in \partial D$ and for every $z,w \in D \cap B(p,\varepsilon_0)$ it holds:
\begin{equation}\label{estimateabove}
K_{D}(w,z) \leq \frac{1}{2}\log\left(1+\frac{|w-z|}{\delta(w)}\right) + \frac{1}{2}\log\left(1+\frac{|w-z|}{\delta(z)}\right)+ C,
\end{equation}
where  $|w-z|$ denotes the Euclidean distance between $w$ and $z$ and $\delta(z)$ the distance from $z$ to $\partial D$.

Moreover, if  $D$ is strongly pseudoconvex, there exist $\varepsilon_1>\varepsilon_0>0$ and $C'>0$ such that for all $p\in \partial D$, $x \not\in D \cap B(p,\varepsilon_1)$ and $z\in D \cap B(p,\varepsilon_0)$ it holds
\begin{equation}\label{estimatebelow}
K_D(x,z) \geq C'-\frac{1}{2}\log(\delta(z)).
\end{equation}
\end{lem}

We start with the following lemma.

\begin{lem}\label{squeeze}
Let $D\subset \C^N$ be a bounded strongly pseudoconvex domain with $C^2$ boundary. Let $V\subset \partial D$ be a closed set and $x\in D$. Then for every open neighborhood $U$ of $V$ there exists $R_0>0$ such that 
\[
\bigcup_{p\in V}F_x(p,R)\subset U\cap D \quad \hbox{for all $0<R\leq R_0$}.
\] 
 In particular, if $V, V'\subset \partial D\setminus\{p\}$ are two closed sets such that $V\cap V'= \emptyset$ then there exists $R_0>0$ such that, for all $0<R<R_0$ it holds $\bigcup_{p\in V}F_x(p,R)\cap \bigcup_{p\in V'}F_x(p,R)=\emptyset$.
\end{lem}

\begin{proof}  Let $R>0$. First, note that
\begin{equation*}
F_x(V, R):=\{z\in D: \liminf_{w\to V}[K_D(z,w)-K_D(x,w)]<\frac{1}{2}\log R\}=\bigcup_{p\in V}F_x(p, R).
\end{equation*}
Indeed, if $z\in F_x(V, R)$, then there exists a sequence $\{w_m\}$ such that 
\[
\lim_{m\to \infty}[K_D(z,w_m)-K_D(x,w_m)]=\liminf_{w\to V}[K_D(z,w)-K_D(x,w)].
\]
Up to subsequences, we can assume that $\{w_m\}$ converges to $p\in V$. Therefore,
\[
\liminf_{w\to p}[K_D(z,w)-K_D(x,w)]\leq \lim_{m\to \infty}[K_D(z,w_m)-K_D(x,w_m)]<\frac{1}{2}\log R,
\]
hence, $z\in F_x(p, R)$. Conversely, if $z\in F_x(p, R)$ then 
\[
\liminf_{w\to V}[K_D(z,w)-K_D(x,w)]\leq \liminf_{w\to p}[K_D(z,w)-K_D(x,w)]<\frac{1}{2}\log R,
\]
proving that $z\in F_x(V,R)$.

Next, we show 
\begin{equation}\label{Eq:intersFV}
\overline{F_x(V, R)}\cap \partial D=V.
\end{equation}
The proof is similar to the proof of Theorem \ref{intersF}.(4) (see, \cite{ABZ, Aba}).  Since $F_x(V, R)=\bigcup_{p\in V}F_x(p, R)$, it follows $V\subseteq \overline{F_x(V, R)}$ by Theorem \ref{intersF}.(4). In order to show the converse inclusion, assume by contradiction there exists $q\in (\overline{F_x(V, R)}\cap \partial D)\setminus V$. Hence, there exists a sequence $\{z_n\}\subset F_x(V, R)$ converging to $q$. By definition, this implies that for every $n\in \N$ there exists a sequence $\{w_{nm}\}_{m\in \N}$ converging to some point $p_n\in V$ such that $[K_D(z_n, w_{nm})-K_D(x, w_{nm})]<\frac{1}{2}\log R$ for every $m\in \N$. Since $V$ is compact, we can assume $\{p_n\}$ is converging to some $p\in V$. Hence, given $\delta>0$, we can also assume that $|z_n-q|<\delta$ and $|w_{nm}-p|<\delta$ for all $n,m$. We can choose $\delta$ so small that $B(q,2\delta)\cap B(p, 2 \delta)=\emptyset$, where $B(q,2\delta):=\{w\in \C^n: |w-q|<2\delta\}$. Hence (see \cite[Cor. 2.3.55]{Aba} or \cite[Thm. 1.6]{ABZ}) there exists $K\in \R$ such that for all $n, m$,
\[
K_D(z_n, w_{nm})\geq -\frac{1}{2}\log \delta(z_n)-\frac{1}{2}\log \delta(w_{nm})+K.
\]
Therefore, by \eqref{estimatebelow},
\begin{equation*}
\begin{split}
\frac{1}{2}\log R&>K_D(z_n, w_{nm})-K_D(x, w_{nm})\geq  -\frac{1}{2}\log \delta(z_n)-\frac{1}{2}\log \delta(w_{nm})+K-K_D(x, w_{nm})\\&\geq  -\frac{1}{2}\log \delta(z_n)-\frac{1}{2}\log \delta(w_{nm})+K-C'+\frac{1}{2}\log \delta(w_{nm})=-\frac{1}{2}\log \delta(z_n)+K-C',
\end{split}
\end{equation*}
a contradiction, and \eqref{Eq:intersFV} holds.

Now, in order to complete the proof, we argue by contradiction. Let assume there exists an open neighborhood $U$ of $V$ such that, for all $n\in \N$ there exists $z_n\in F_x(V,\frac{1}{n})\cap (D\setminus U)$. Up to subsequence, we can assume that $z_n\to z_0\in \overline{D\setminus U}$.  
Arguing as in the proof of Proposition \ref{outsideK}, it is not hard to see that $z_0\not\in D$. Therefore, $z_0\in \partial D$. Fix $R>0$. Since $F_x(V,R)\subset F_x(V,R')$ for all $0<R<R'$, it follows that $z_n\in F_x(V,R)$ for all $n>1/R$. In particular, $z_0\in \overline{F_x(V,R)}\cap \partial D=V$ by \eqref{Eq:intersFV}. Again, a contradiction.
\end{proof}

We prove the following optimal localization principle for the Kobayashi distance on strongly pseudoconvex domains. An associated result was given by Z. Balongh and M. Bonk in \cite{BB}. Our more precise statement relies on their approach. After a first draft of our paper was completed, N. Nikolov informed us that one can  obtain a similar estimate using the techniques he developed in \cite{N}.

\begin{lem}\label{localizationL}
Let $D\subset \C^N$ be a bounded strongly pseudoconvex domain with $C^2$ boundary. Let $p\in \partial D$. Let $U$ be an open neighborhood of $p$. Then there exists an open neighborhood $W\subset U$ of $p$ and a constant $T\geq 1$ such that for every $z,w\in W\cap D$ it holds
\begin{equation}\label{localization}
K_{U\cap D}(w,z)-K_D(w,z)\leq \frac{1}{2}\log T.
\end{equation} 
\end{lem}
\begin{proof}
First of all, notice that, given any open set $U$ of $p$, there exists an open neighborhood $\tilde{U}\subset U$ of $p$ such that $\tilde{U}\cap D$ is a strongly pseudoconvex domain with $C^2$ boundary. If \eqref{localization}  holds for $\tilde{U}$, then for every $z,w\in W\cap D$ 
\[
K_{U\cap D}(w,z)-K_D(w,z)\leq K_{\tilde{U}\cap D}(w,z)-K_D(w,z)\leq \frac{1}{2}\log T,
\]
and hence \eqref{localization}ÃÂ  holds for $U$ as well.
 
Therefore, without loss of generality, we can assume  that $U\cap D$ is a strongly pseudoconvex domain with $C^2$ boundary.

For $x,y\in \partial D$, let $d^{\partial D}_H(x,y)$ denote the Carnot-Carath\'eodory distance between $x$ and $y$. The distance $d^{\partial D}_H(x,y)$ is defined as follows (see, {\sl e.g.}, \cite{BB}). A piecewise $C^1$-smooth curve  $\alpha:[0,1]\to  \partial D$ is {\sl horizontal} provided $\alpha'(t)\in T^{\C}_{\alpha(t)}\partial D$ for almost every $t$. The set of horizontal curves is denoted by $\mathcal H(\partial D)$. For every $x,y\in \partial D$ there exist horizontal curves joining $x$ and $y$. Let 
\[
\rho_D(z):=\begin{cases} -\delta(z)& \quad \hbox{for } z\in D\\
\delta(z)& \quad \hbox{for } z\in \C^N\setminus D
\end{cases}
\]
Then $\rho_D$ is $C^2$ on an open neighborhood of $\partial D$. Let $L_{\rho_D}$ be the Levi form of $\rho_D$. The Levi-length of a horizontal curve $\alpha$ is defined as
\[
\ell^D_L(\alpha):=\int_0^1 (L_{\rho_D}(\alpha(t); \alpha'(t)))^{1/2}dt.
\]
Then 
\[
d^{\partial D}_H(x,y):=\inf \{ \ell^D_L(\alpha): \alpha\in \mathcal H(\partial D), \alpha(0)=x, \alpha(1)=y\}.
\]
Let $W'$ be an open neighborhood of $p\in \partial D$ such that for every $z\in W'\cap D$ there exists a unique point $\pi_D(z)\in \partial D$ such that 
\[
\delta(z)=|z-\pi_D(z)|.
\]
For $z,w\in W'\cap D$ define
\[
g_D(z,w):=2\log \left[\frac{d_H^{\partial D}(\pi_D(z),\pi_D(w))+ \max\{\sqrt{\delta(z)}, \sqrt{\delta(w)}\}}{\sqrt{\delta(z)\delta(w)}}\right].   
\]
By \cite[Corollary 1.3]{BB} there exists $C_D\geq 0$ such that for all $z,w\in W'\cap D$ it holds
\begin{equation}\label{bb1}
g_D(z,w)-C_D\leq K_D(z,w)\leq g_D(z,w)+C_D.
\end{equation}
Similarly, there exists $C_{U\cap D}\geq 0$ such that for all $z,w\in W'\cap D$ it holds
\begin{equation}\label{bb2}
g_{U\cap D}(z,w)-C_{U\cap D}\leq K_{U\cap D}(z,w)\leq g_{U\cap D}(z,w)+C_{U\cap D}.
\end{equation}
By \eqref{bb1} and \eqref{bb2}, for all $z,w\in W'\cap D$ we have
\begin{equation}\label{Kandg}
K_{U\cap D}(w,z)-K_D(w,z)\leq g_{U\cap D}(z,w)-g_D(z,w)+C_D+C_{U\cap D}.
\end{equation}
Now, up to shrinking $W'$ is necessary, we can assume that for all $z\in W'\cap D$ it holds
\[
\delta^{U\cap D}(z)=\delta(z),
\]
where $\delta^{U\cap D}(z)$ denotes the distance from $z$ to the boundary of $U\cap D$.
With such an assumption, $\pi(z):=\pi_D(z)=\pi_{U\cap D}(z)$. We claim that there exists a open neighborhood $W\subseteq W'$ such that for all $z,w\in W\cap D$ it holds
\begin{equation}\label{carnot}
d_H^{\partial D}(\pi(z),\pi(w))=d_H^{\partial (U\cap D)}(\pi(z),\pi(w)).
\end{equation}
Once this is proved, it follows from \eqref{Kandg} that for all $z,w\in W\cap D$
\[
K_{U\cap D}(w,z)-K_D(w,z)\leq \frac{1}{2}\log T:=C_D+C_{U\cap D},
\]
and the lemma is proved.

In order to prove \eqref{carnot}, we argue as follows. By the ``box-ball estimate'' (see \cite[eq. (3.1)]{BB}) there exist $A_1, A_2>0$ such that
\[
A_1|x-y|\leq d_H^{\partial D}(x,y)\leq A_2|x-y|^{1/2},
\]
for all $x,y\in \partial D$. In particular this implies that there exists   an open neighborhood $W_1\subset W'$ of $p$ such that for every $x,y\in W_1\cap \partial D$ and for every  horizontal curve $\alpha$ joining $x$ and $y$ such that $\alpha([0,1])\cap (\partial D\setminus W')\neq\emptyset$ it holds
\begin{equation}\label{out1}
\ell^D_L(\alpha)\geq d^{\partial D}_H(x,y)+c_1,
\end{equation}
for some $c_1>0$. Similarly, that there exists   an open neighborhood $W_2\subset W'$ of $p$ such that for every $x,y\in W_2\cap \partial D$ and for every  horizontal curve $\alpha$ joining $x$ and $y$ such that $\alpha([0,1])\cap (\partial (U\cap D)\setminus W')\neq\emptyset$ it holds
\begin{equation}\label{out2}
\ell^{U\cap D}_L(\alpha)\geq d^{\partial {(U\cap D)}}_H(x,y)+c_2,
\end{equation}
for some $c_2>0$. Let $W\subset W_1\cap W_2$ be an open neighborhood of $p$ such that $\pi(W\cap D)\subset\subset (W_1\cap W_2)\cap \partial D$. Let $z,w\in W\cap  D$. Let $0<\epsilon<\min\{c_1,c_2\}$. Let $\alpha$ be a horizontal curve joining $\pi(z)$ and $\pi(w)$ such that 
\[
\ell^D_L(\alpha)\leq d^{\partial D}_H(\pi(z),\pi(w))+\epsilon.
\]
 By \eqref{out1}, $\alpha([0,1])\subset W'$. Since $\rho_D(z)=\rho_{U\cap D}(z)$ for all $z\in W'\cap D$, it follows that $\ell^{U\cap D}_L(\alpha)=\ell^D_L(\alpha)$. Hence 
\[
d_H^{\partial (U\cap D)}(\pi(z),\pi(w))\leq \ell^{U\cap D}_L(\alpha)=\ell^D_L(\alpha)\leq d^{\partial D}_H(\pi(z),\pi(w))+\epsilon,
\]
and by the arbitrariness of $\epsilon$, it follows that $d_H^{\partial (U\cap D)}(\pi(z),\pi(w))\leq d^{\partial D}_H(\pi(z),\pi(w))$. A similar argument gives the opposite inequality, and the proof is completed.
\end{proof}

\begin{defi}
Let $D\subset \C^N$ be a domain. A {\sl cone region} $\mathcal C(p,\alpha, \epsilon)$ in $D$ of vertex $p\in \partial D$, aperture $\alpha>1$ and size $\epsilon>0$ is 
 \[
 \mathcal C(p,\alpha, \epsilon):=\{z\in D: |z-p|<\min\{\alpha \delta(z), \epsilon\}\}.
 \]
\end{defi}

\begin{prop}\label{pointstrict}
Let $D\subset \C^N$ be a bounded strongly pseudoconvex domain with $C^2$ boundary, $x\in D$. Let $\alpha>1$ and $R>0$. Then there exists $\varepsilon>0$ such that for every $p\in \partial D$ and for every sequence $\{u_n\}$ converging to $p$ it holds \begin{equation}\label{coneeq}
 \mathcal C(p,\alpha, \varepsilon)\subset E_x(\{u_n\}, R).
\end{equation}
 In particular, every sequence in $D$ converging to a boundary point is admissible.
\end{prop}

\begin{proof}
Let $\{u_n\}\subset D$ be a sequence converging to $p \in \partial D$.  

Let $\varepsilon_1>\varepsilon_0>0$ be given by  \eqref{estimateabove} and \eqref{estimatebelow} and  such that  $x \not\in D \cap B(p,\varepsilon_1)$. Then for $w \in D \cap B(p,\varepsilon_0)$ and $n$ large enough, we have:

\begin{equation}\label{kob-equation}
K_D(w,u_n) - K_D(x,u_n)\leq \frac{1}{2}\log\left(1+\frac{|u_n-w|}{\delta(w)}\right) + \frac{1}{2}\log (\delta(u_n) + |u_n-w|) + C - C'.
\end{equation}
Therefore,
\begin{equation}\label{limkob}
\limsup_{n\to \infty} [K_D(w,u_n) - K_D(x,u_n)] \leq C-C'+\frac{1}{2}\log \left(1+\frac{|p-w|}{\delta(w)} \right)+\frac{1}{2}\log|p-w|.
\end{equation}

Now,  fix $\alpha > 1$. Given $\varepsilon\in (0, \varepsilon_0]$, for every $w\in   \mathcal C(p,\alpha, \varepsilon)$ it holds
\[
\limsup_{n\to \infty} [K_D(w,u_n) - K_D(x,u_n)] \leq C-C'+\frac{1}{2}\log \left(1+\alpha \right)+\frac{1}{2}\log \varepsilon.
\]
Hence, given $R>0$, there exists $\varepsilon\in (0,\varepsilon_0]$  (depending on $R$ and $\alpha$ but not on $p$ and $\{u_n\}$), such that \eqref{coneeq} holds. 

In particular,  $E_x(\{u_n\},R)\neq \emptyset$ for all $R>0$.   Since $D$ is complete hyperbolic, it follows that $\lim\inf_{n \rightarrow \infty}K_M(x,u_n) = \infty$. Hence $\{u_n\}$ is admissible.
\end{proof}





For strongly pseudoconvex domains we can say more about equivalent admissible sequence:

\begin{prop}\label{stronglypseudo}
Let $D\subset \C^N$ be a bounded strongly pseudoconvex domain with $C^3$ boundary. If $\{u_n\}, \{v_n\}\subset D$ are two sequences converging to $p$ then they are admissible and $\{u_n\} \sim \{v_n\}$. 
\end{prop}

\begin{proof}
The two sequences are admissible by Proposition \ref{pointstrict}.

Fix $R>0$. We want to show that there exist $R', R''>0$ such that $E_x(\{u_n\}, R')\subset E_x(\{v_n\}, R)$ and $E_x(\{v_n\}, R'')\subset E_x(\{u_n\}, R)$. We are going to prove the existence of $R'$, a similar argument holds for $R''$.

We claim that for every open neighborhood $W$ of $p$, there exists $r_0>0$ such that 
\begin{equation}\label{include}
\overline{E_x(\{u_n\}, r)}\subset (W\cap D)\cup\{p\},
\end{equation} 
for all $0<r<r_0$ (and similarly for $\{v_n\}$). Indeed, for every $r>0$, $E_x(\{u_n\}, r)\subset F_x(p,r)$, where $F_x(p,r)$ denotes  Abate's big horosphere.  Hence, the claim follows  from Lemma \ref{squeeze}. 

Now choose $U$ to be an open neighborhood of $p$ such that $U\cap D$ is biholomorphic to a strongly convex domain with $C^3$ boundary. Let $W\subset U$ be the open neighborhood of $p$ given by Lemma \ref{localizationL}.

Since the equivalence relation among admissible sequences is independent of the base point, we can assume with no loss of generality that $x\in W\cap D$. Moreover, since both $\{u_n\}$ and $\{v_n\}$ are eventually contained in $W$, we can also assume that $\{u_n\}, \{v_n\}\subset W\cap D$. 

For $r>0$, let $E^{U\cap D}_x(p,r)$ denote the small horosphere of vertex $p$ and radius $r$ in $U\cap D$. Since clearly $E^{U\cap D}_x(p,r)\subset E^{U\cap D}_x(\{u_n\}, r)\subset F^{U\cap D}_x(p,r)$ and by Theorem \ref{intersF}, $E^{U\cap D}_x(p,r)=F^{U\cap D}_x(p,r)$, it follows that for every $r>0$
\begin{equation}\label{equahoro}
E^{U\cap D}_x(p,r)=E^{U\cap D}_x(\{u_n\},r)=E^{U\cap D}_x(\{v_n\},r).
\end{equation}

Now we claim that for every $r>0$ such that $E_x(\{v_n\},r)\subset W$ there exists $r'>0$ such that 
\begin{equation}\label{include1}
E^{U\cap D}_x(p,r')\subset E_x(\{v_n\},r).
\end{equation}
Conversely,  for every $r>0$ such that $E^{U\cap D}_x(p,r)\subset W$ there exists $r'>0$ such that 
\begin{equation}\label{include2}
E_x(\{u_n\},r')\subset E^{U\cap D}_x(p,r).
\end{equation}
Assuming the claims, the proof ends as follows. By \eqref{include}, there exists $r>0$ such that $E_x(\{v_n\}, r)\subset E_x(\{v_n\}, R)\cap W$. By \eqref{include1} there exists $r'>0$ such that $E^{U\cap D}_x(p,r')\subset E_x(\{v_n\},r)$, and by \eqref{include2} there exists $R'>0$ such that $E_x(\{u_n\},R')\subset E^{U\cap D}_x(p,r')$. Hence, $E_x(\{u_n\},R')\subset E_x(\{v_n\}, R)$ as needed. 

In order to prove \eqref{include1},  let $r>0$ be such that $E_x(\{u_n\},r)\subset W$. Let $T\geq 1$ be given by Lemma \ref{localizationL}.  Let $r_T\in (0,r]$ be such that $E^{U\cap D}_x(p,T^{-1}r_T)\subset W$. By \eqref{equahoro}, $E^{U\cap D}_x(p,T^{-1}r_T)=E^{U\cap D}_x(\{v_n\},T^{-1}r_T)$. Hence, for every $z\in E^{U\cap D}_x(p,T^{-1}r_T)$ it holds
\begin{equation*}
\begin{split}
&\limsup_{n\to \infty} [K_D(z,v_n)-K_D(x,v_n)]\leq \limsup_{n\to \infty} [K_{U\cap D}(z,v_n)-K_{D}(x,v_n)]\\ &\leq \limsup_{n\to \infty} [K_{U\cap D}(z,v_n)-K_{U\cap D}(x,v_n)]+\limsup_{n\to \infty}[K_{U\cap D}(x,v_n)-K_D(x,v_n)]\\ &\stackrel{\eqref{localization}}{\leq} \frac{1}{2}\log (T^{-1}r_T)+\frac{1}{2}\log T=\frac{1}{2}\log r_T,
\end{split}
\end{equation*}
which shows that $z\in E_x(\{v_n\},r_T)\subseteq E_x(\{v_n\},r)$. Thus, $E^{U\cap D}_x(p,T^{-1}r_T)\subset E_x(\{v_n\},r)$ and \eqref{include1} is proved with $r'=T^{-1}r_T$.

In order to prove \eqref{include2},  let $r>0$ be such that  $E^{U\cap D}_x(p,r)\subset W$. Let $T\geq 1$ be given by Lemma \ref{localizationL}.  Let $r_T\in (0,r]$ be such that $E_x(\{u_n\},T^{-1}r_T)\subset W$.  Then for every $z\in E_x(\{u_n\},T^{-1}r_T)$ it holds
\begin{equation*}
\begin{split}
&\limsup_{n\to \infty} [K_{U\cap D}(z,u_n)-K_{D\cap U}(x,u_n)]\leq \limsup_{n\to \infty} [K_{U\cap D}(z,u_n)-K_{D}(x,u_n)]\\ &\leq \limsup_{n\to \infty} [K_{U\cap D}(z,u_n)-K_{D}(z,u_n)]+\limsup_{n\to \infty}[K_D(z,u_n)-K_D(x,u_n)]\\ & \stackrel{\eqref{localization}}{\leq} \frac{1}{2}\log T+\frac{1}{2}\log (T^{-1}r_T)=\frac{1}{2}\log r_T,
\end{split}
\end{equation*}
which, by \eqref{equahoro}, shows that $z\in E^{U\cap D}_x(p,r_T)\subseteq E^{U\cap D}_x(p,r)$. Thus, $E_x(\{u_n\},T^{-1}r_T)\subset E^{U\cap D}_x(p,r)$, and \eqref{include2} is proved with $r'=T^{-1}r_T$.
\end{proof}

A consequence of the previous proposition is given by the following:

\begin{prop}\label{convergenceadmsp}
Let $D\subset \C^N$ be a bounded strongly pseudoconvex domain with $C^3$ boundary. Let $\{u_n\}\subset D$ be an admissible sequence. Then there exists $p\in \partial D$ such that $\lim_{n\to \infty} u_n=p$. Moreover, every admissible sequence $\{v_n\}$ which is equivalent to $\{u_n\}$ converges to $p$. Finally, for $p\in \partial D$, let  $\mathcal A_p$ denote the set of all  sequences in $D$ which converges to $p$. Then, for all $R>0$
\begin{enumerate}
\item $F_x(p,R)=\cup_{\{u_n\}\in \mathcal A_p} E_x(\{u_n\}, R)$;
\item $E_x(p,R)=\cap_{\{u_n\}\in \mathcal A_p} E_x(\{u_n\}, R)\neq \emptyset$.
\end{enumerate}
\end{prop}

\begin{proof}
Since $D$ is complete hyperbolic, Property (1) in Definition \ref{admis-def} implies that all accumulation points of $\{u_n\}$ are contained in $\partial D$. Let $p\in \partial D$ be one of such points. Let $\{u_{n_k}\}$ be a subsequence of $\{u_n\}$ such that $u_{n_k}\to p$. Let $z\in E_x(\{u_n\}, R)$ for some $R>0$. Then
\begin{equation*}
\begin{split}
\liminf_{w\to p} [K_D(z,w)-K_D(x,w)]&\leq \limsup_{k\to \infty}[K_D(z,u_{n_k})-K_D(x,u_{n_k})]\\
&\leq \limsup_{n\to \infty}[K_D(z,u_{n})-K_D(x,u_{n})]\leq \frac{1}{2}\log R.
\end{split}
\end{equation*}
Hence $z\in F_x(p,R)$. This implies that $E_x(\{u_n\}, R)\subset F_x(p,R)$. In particular, if $q\in \partial D$ were another accumulation point of $\{u_n\}$ different from $p$, it would hold 
\[
E_x(\{u_n\}, R)\subset F_x(p,R)\cap F_x(q,R)\quad \forall R>0,
\]
contradicting Lemma \ref{squeeze}. Hence $\lim_{n\to \infty}u_n=p$.

Now, let $\{v_n\}$ be another admissible sequence which is equivalent to $\{u_n\}$. From what we just proved, there exists $q\in \partial D$ such that $\lim_{n\to q}v_n=q$. If $q\neq p$, let $R_0>0$ be given by Lemma \ref{squeeze} such that $F_x(p,R)\cap F_x(q,R)=\emptyset$ for all $0<R<R_0$. Fix $0<R<R_0$. By the previous argument, $E_x(\{v_n\}, R)\subset F_x(q,R)$. Since $\{u_n\}$ is equivalent to $\{v_n\}$ this implies that there exists $R'>0$ such that $E_x(\{u_n\}, R')\subset E_x(\{v_n\}, R)\subset F_x(q, R)$. Now, if $R'\geq R$, it follows that $E_x(\{u_n\}, R)\subset F_x(q, R)$, giving a  contradiction since we proved that $E_x(\{u_n\}, R)\subset F_x(p, R)$. If $R'<R$, again we obtain a contradiction since  $E_x(\{u_n\}, R')\subset E_x(\{u_n\}, R)\subset F_x(p, R)$.

In order to prove (1), we already saw that $\cup_{\{u_n\}\in \mathcal A_p} E_x(\{u_n\}, R)\subset F_x(p,R)$. Let $z\in F_x(p,R)$. By the very definition, there exists a sequence $\{u_n\}\subset D$ converging to $p$ such that $\lim_{n\to \infty}[K_D(z,u_n)-K_D(x,u_n)]=\frac{1}{2}\log R'$ for some $R'<R$. By Proposition \ref{stronglypseudo}, $\{u_n\}$ is admissible and $z\in E_x(\{u_n\}, R)$, hence $F_x(p,R)\subset \cup_{\{u_n\}\in \mathcal A_p} E_x(\{u_n\}, R)$. 

(2) the proof is similar and we omit it. The fact that $E_x(p,R)\neq\emptyset$ follows then immediately from \eqref{coneeq}.
\end{proof}

The main result of this section is the following:

\begin{theo}\label{main-pseudo}
Let $D\subset \C^N$ be a bounded strongly pseudoconvex domain with $C^3$ boundary. Then $\hat{D}$ endowed with the horosphere topology  is homeomorphic to $\overline{D}$ (closure of $D$ in $\C^N$) endowed with the Euclidean topology.
\end{theo}

\begin{proof}
We define a map $\Theta: \hat{D} \to \overline{D}$ as follows. If $z\in D\subset \hat{D}$ then we just set $\Theta(z)=z$. If $\underline{y}\in \partial_H D$, by Proposition \ref{convergenceadmsp}  all admissible sequences representing $\underline{y}$ converge to a same point $p\in \partial D$. Then we set $\Theta(\underline{y}):=p$. The map $\Theta$ is bijective by Proposition \ref{stronglypseudo}. 

In order to prove that $\Theta$ is a homeomorphism, we show that given  $C\subset \overline{D}$ then $C$ is closed in the Euclidean topology if and only if $\Theta^{-1}(C)\subset \hat{D}$ is closed in the horosphere topology.

{\sl Assume that $C$ is closed in the Euclidean topology}. In order to show that $\Theta^{-1}(C)$ is closed in the horosphere topology, we need to show that if $\{\xi_n\}$ is a sequence in $\Theta^{-1}(C)$ converging to $\xi\in \hat{D}$, then $\xi\in \Theta^{-1}(C)$. There are three cases to examine, from which the general case follows immediately:

{\sl Case 1.} If $\{\Theta^{-1}(z_n)\}$ is a sequence in $\Theta^{-1}(C)\cap D$ converging to $\Theta^{-1}(z)\in D$, then by the very definition of $\Theta$ and since the topology induced on $D\subset\hat{D}$ by the horosphere topology is the Euclidean topology, it follows easily that $\Theta^{-1}(z)\in \Theta^{-1}(C)$.

{\sl Case 2.} Let now $\{\underline{y}_m\}$ be a sequence in $\Theta^{-1}(C)\cap \partial_H D$, converging in the horosphere topology to $\underline{y}\in \partial_H D$ (by Remark \ref{boundaryconv} a sequence on the horosphere boundary can only converge to a point on the horosphere boundary). Let $p_m:=\Theta(\underline{y}_m)\in C$, $m\in \N$ and let $p:=\Theta(\underline{y})$. By the definition of convergence in the horosphere topology, there exist horosphere sequences $\{u_n^m\}_{n\in \N}$ with $[u_n^m]=\underline{y}_m$ and $\{u_n\}$ with $[u_n]=\underline{y}$ such that for every $R>0$ there exists $m_R\in \N$ such that $E_x(\{u_n^m\}, R)\cap E_x(\{u_n\}, R)\neq \emptyset$ for $m>m_R$. By the definition of $\Theta$, it holds  $\lim_{n\to \infty} u_n^m=p_m$ and $\lim_{n\to \infty}u_n=p$. Thus, by Proposition \ref{convergenceadmsp}.(1), it follows that 
\[
F_x(p, R)\cap F_x(p_m, R)\neq\emptyset
\]
 for all $m>m_R$. If $\{p_m\}$ contained a subsequence not converging to $p$, say $\{p_{m_k}\}$,  taking $V:=\{p\}$ and $V':=\overline{\{p_{m_k}\}}$,  we would get a contradiction with Lemma \ref{squeeze}. Therefore,   $\lim_{m\to \infty} p_m=p$. Since $C$ is closed, this implies that $p\in C$ as well, hence $\Theta^{-1}(p)=\underline{y}$ belongs to $\Theta^{-1}(C)$.
 
{\sl Case 3.} Finally, assume $\{z_n=\Theta^{-1}(z_n)\}\subset \Theta^{-1}(C)\cap D$ converges to $\underline{y}\in \partial_H D$. Let $p=\Theta(\underline{y})\in \partial D$.  By Definition \ref{conv-int-out}, and by Proposition \ref{convergenceadmsp}.(1), it follows that there exists a sequence $\{p_n\}\subset \partial D$ such that for all $R>0$ there exists $n_R\in \N$   such that $z_n\in F_x(p_n, R)$ and $F_x(p_n, R)\cap F_x(p, R)\neq \emptyset$ for $n\geq n_R$. As before, this implies that $\lim_{n\to \infty} p_n=p$. Thus, given an open set $U$ containing $p$, there exists $J\in \N$ such that  $\{p_n\}_{n\geq J}\subset U$. Let $V=\{p_n\}_{n\geq J}\cup \{p\}$. The set $V$ is closed, and by Lemma \ref{squeeze}, there exists $R_0>0$ such that  $\bigcup_{n\geq J} F_x(p_n, R)\subset U\cap D$ for all $0<R<R_0$. Therefore, $\{z_n\}$ is eventually contained in $U$. By the arbitrariness of $U$, it follows that $\lim_{n\to \infty}z_n=p$. Since $C$ is closed, this implies that $p\in C$ and hence $\underline{y}=\Theta^{-1}(p)\in \Theta^{-1}(C)$. 
 
Hence, if $C$ is closed in the Euclidean topology then $\Theta^{-1}(C)$ is closed in the horosphere topology.
 
{\sl Assume that $\Theta^{-1}(C)$ is closed in the horosphere topology}. We have to show that if $\{\xi_n\}\subset C$ is a sequence converging to $\xi\in \overline{D}$ then $\xi \in C$. As before, we distinguish three cases.

{\sl Case 1.} If $\{\xi_n\}, \xi\subset D$, then the statement is clearly true. 

{\sl Case 2.} Let $\{p_m\}$ be a sequence in $C\cap \partial D$ converging to $p\in \partial D$. Let $\underline{y}_m:=\Theta^{-1}(p_m)\in \Theta^{-1}(C)\cap \partial_H D$ and let $\underline{y}:=\Theta^{-1}(p)\in \partial_H D$. 

We want to show that $\{\underline{y}_m\}$ converges to $\underline{y}$ in the horosphere topology. Once we proved this, since $\Theta^{-1}(C)$ is closed, it follows that $\underline{y}\in \Theta^{-1}(C)$ and hence, $p=\Theta(\underline{y})\in C$.

To show this, for every $m\in \N$, let $\{v_n^m\}_{n\in \N}$ be a sequence converging to $p_m$,  and let $\{v_n\}$ be a sequence converging to $p$.  By Proposition \ref{convergenceadmsp}, $[\{v_n^m\}]=\underline{y}_m$ and $[\{v_n\}]=\underline{y}$. Fix $R>0$. By \eqref{coneeq}, there exists $\varepsilon>0$ such that the c\v{o}ne region $\mathcal C(q, 2, \varepsilon)\subset E_x(\{w_n\}, R)$ for all $q\in \partial D$ and all sequences $\{w_n\}$ converging to $q$. Since $\partial D$ is $C^1$, there exists an open neighborhood $V$ of $p$ such that for all $q\in V\cap \partial D$ it holds
$\mathcal C(p, 2, \varepsilon)\cap  \mathcal C(q, 2, \varepsilon)\neq \emptyset$. Since  $\{p_m\}\subset V$ eventually, it follows that $E_x(\{v_n^m\}, R)\cap E_x(\{v_n\}, R)\neq \emptyset$ for all $m$ sufficiently large. By definition,  this means that $\{\underline{y}_m\}$ converges to $\underline{y}$ in the horosphere topology, as needed.

{\sl Case 3.} Let $\{z_m\}$ be a sequence in $C\cap  D$ converging to $p\in \partial D$. Let $\underline{y}:=\Theta^{-1}(p)$. As before, to conclude that $p\in C$, it is enough to show that $\{z_m=\Theta^{-1}(z_m)\}$ converges to $\underline{y}$ in the horosphere topology. To show this, fix $\alpha>1$ and for $N\in \N$, let $\varepsilon_N>0$ be given by Proposition \ref{pointstrict} such that for every $q\in \partial D$ and for every sequence $\{w_n\}$ converging to $q$ it holds $\mathcal C(q,\alpha, \varepsilon_N)\subset E_x(\{w_n\}, \frac{1}{N})$.

Since $\{z_m\}$ converges to $p$ and $\partial D$ is $C^2$, we can assume with no loss of generality that for all $m\in \N$ there exists a unique $p_m\in \partial D$ such that $|z_m-p_m|=\delta(z_m)$. Note that $z_m\in \mathcal C(p_m,\alpha, r)$ for all $r>\delta(z_m)$. 

For each $m\in \N$ let $\{v_n^m\}\subset D$ be a sequence converging to $p_m$ and let $\{u_n\}\subset D$ be a sequence converging to $p$. By Proposition \ref{convergenceadmsp},  $[\{u_n\}]=\underline{y}$.

Fix $N\in \N$. Since $\delta(z_m)\to 0$,  there exists $m^1_N\in \N$ such that $z_m \in \mathcal C(p_m,\alpha, \varepsilon_N)\subset E_x(\{v_n^m\}, \frac{1}{N})$ for all $m\geq m^1_N$. Moreover, since $\{p_m\}$ converges to $p$, there exists $m_N^2\in \N$ such that $\mathcal C(p_m,\alpha, \varepsilon_N)\cap C(p,\alpha, \varepsilon_N)\neq \emptyset$ for all $m\geq m^2_N$, which implies $E_x(\{v_n^m\}, \frac{1}{N})\cap E_x(\{u_n\}, \frac{1}{N})\neq \emptyset$ for all $m\geq m^2_N$. 

Now let $R>0$ be given. Let $N\in \N$, $N>1/R$ and let $m_R:=\max\{m^1_N, m^2_N\}$. Then for all $m\geq m_R$ it holds $z_m\in  E_x(\{v_n^m\}, \frac{1}{N})\subset E_x(\{v_n^m\}, R)$ and 
\[
\emptyset \neq E_x(\{v_n^m\}, \frac{1}{N})\cap E_x(\{u_n\}, \frac{1}{N})\subset E_x(\{v_n^m\}, R)\cap E_x(\{u_n\}, R),
\]
which means that $\{z_m\}$ converges to $\underline{y}$ in the horosphere topology, and we are done.

Hence, if $\Theta^{-1}(C)$ is closed in the horosphere topology then $C$ is closed in the Euclidean topology.
\end{proof}

We end this section by proving the following result which will be useful to study boundary behavior:

\begin{prop}\label{strongly-pseudo-boundary}
Let $D\subset \C^N$ be a bounded strongly pseudoconvex domain with $C^3$ boundary. Let $\underline{x}\in \partial_H D$. Then there exists $p_{\underline{x}}\in \partial D$ such that  $\hbox{ I}^H_D(\underline{x})=\{p_{\underline{x}}\}$. Moreover, if $\underline{y}\in \partial_H D$ then $p_{\underline{x}}=p_{\underline{y}}$ if and only if $\underline{x}= \underline{y}$.
 \end{prop}
\begin{proof}
By Theorem \ref{main-pseudo}, there is a homeomorphism $\Theta: \hat{D}\to \overline{D}$ (where $\overline{D}$ is the closure of $D$ in $\C^N$). Let $p_{\underline{x}}:=\Theta(\underline{x})$. Recall also that $\Theta(z)=z$ for $z\in D$.

We have to show that, if $\{w_n\}\subset D$ is a sequence converging to $\underline{x}$ in the horosphere topology, then $\{w_n\}$ converges to $p_{\underline{x}}$ in the Euclidean topology. This follows at once by Theorem \ref{main-pseudo}, because  $\{w_n\}$ converges to $\underline{x}$ in the horosphere topology if and only if $\{\Theta(w_n)=w_n\}$ converges $p_{\underline{x}}$ in the Euclidean topology. This also shows the last statement of the proposition.
\end{proof}

\section{Convex domains}\label{convexsec}

In this section we consider convex domains in $\mathbb C^n$. Here we mean convex in the real geometrical sense, that is, $D\subset \C^N$ is {\sl convex }ÃÂ if for every two points $p,q\in D$ the real segment $[p,q]$ joining $p$ and $q$ is contained in  $D$. 

By \cite{BS}, a (possibly unbounded) convex domain in $\C^N$ is complete hyperbolic if and only if it is biholomorphic to a bounded domain of $\C^N$, in particular, a convex domain is hyperbolic if and only if it is complete hyperbolic. 

We start with the following result:

\begin{prop}\label{horo-convex}
Let $D\subset \C^N$ be a  hyperbolic convex domain. Let $x\in D$. Let $\{u_n\}$ be an admissible sequence in $D$. Then for every $R>0$ the horosphere $E_x(\{u_n\}, R)$ is convex.
\end{prop}

\begin{proof}
Let $B(0,N):=\{z\in \C^N : |z|<N\}$, $N\in \N$.  Let $D_N:=D\cap B(0,N)$. Then $D_N$ is a bounded convex domain and its Kobayashi distance $k_{D_N}$ is a convex function (see \cite[Prop. 2.3.46]{Aba}). Passing to the limit, it turns out that $k_D$ is a convex function as well. 

Now, let $\{u_n\}\subset D$ be an admissible sequence. Let $R>0$ and let $z,w \in E_x(\{u_n\}, R)$. Then, for $s\in (0,1)$ we have
\begin{equation*}
K_D(sz+(1-s)w,u_n)-K_D(x,u_n) 
\leq \max\{K_D(z,u_n), K_D(w,u_n)\}-K_D(x,u_n).
\end{equation*}
Taking the limsup as $n\to \infty$, this implies that $sz+(1-s)w\in E_x(\{u_n\}, R)$, which is thus convex.
\end{proof}

\begin{defi}
Let $D\subset \C^N$ be a bounded convex domain, $x\in D$. Let $\{u_n\}$ be an admissible sequence. For $R>0$ let 
\[
\hbox{II}(\{u_n\}, R):=\overline{E_x(\{u_n\}, R)}\cap \partial D. 
\]
\end{defi}
Note that, if $\{u_n\}$ is an admissible sequence and denoting $[\{u_n\}]\in \partial_HD$, by Lemma \ref{II-to-hor} it holds
\[
\hbox{II}_D^H([\{u_n\}])=\bigcap_{R>0}\overline{E_x(\{u_n\}, R)}=\bigcap_{R>0} \hbox{II}(\{u_n\}, R).
\]

\begin{rem}\label{simple}
Note that $\hbox{II}(\{u_n\}, R)$ and $\hbox{II}_D^H([\{u_n\}])$ are nonempty   sets and $\hbox{II}_D^H([\{u_n\}])$ is convex.  Indeed, this follows at once from Proposition \ref{horo-convex}, Proposition \ref{propertyhoro}.(2) and since $\overline{E_x(\{u_n\}, R)}\subset \overline{E_x(\{u_n\}, R')}$ for all $0<R<R'$. Moreover, while $\hbox{II}(\{u_n\}, R)$ depends in general on the base point $x$, the horosphere principal part, $\hbox{II}_D^H([\{u_n\}])$, does not.
\end{rem}

\begin{lem}\label{whois}
Let $D\subset \C^N$ be a bounded convex domain. Let $\{u_n\}$ be an admissible sequence.  Let  $\mathcal M$ denote the set of all convergent  subsequences extracted from $\{u_n\}$. Then
\begin{equation}\label{intersI}
\hbox{II}_D^H([\{u_n\}])=\bigcap_{\{v_n\}\in \mathcal M} \hbox{II}_D^H([\{v_n\}]).
\end{equation}
Moreover, if $\{u_n\}$ is convergent to $p\in \partial D$, then $p\in \hbox{II}_D^H([\{u_n\}])$.
\end{lem}

\begin{proof}
Let $R>0$ and let $z\in E_x(\{u_n\}, R)$. For every $\{v_n\}\in \mathcal M$,
\[
\limsup_{n\to \infty}[K_D(z,v_n)-K_D(x, v_n)]\leq \limsup_{n\to \infty}[K_D(z,u_n)-K_D(x, u_n)]<\frac{1}{2}\log R,
\]
proving that $E_x(\{u_n\}, R)\subseteq\bigcap_{\{v_n\}\in \mathcal M} E_x(\{v_n\}, R)$. Conversely, assume $z\in \bigcap_{\{v_n\}\in \mathcal M} E_x(\{v_n\}, R)$. We can find a converging subsequence $\{u_{n_k}\}$ of $\{u_n\}$  such that
\[
\lim_{k\to \infty}[K_D(z,u_{n_k})-K_D(x, u_{n_k})]=\limsup_{n\to\infty}[K_D(z,u_n)-K_D(x, u_n)]. 
\]
Since $\{u_{n_k}\}\subset \mathcal M$, it follows that $z\in E_x(\{u_{n_k}\}, R)$, and the previous equation implies $z\in E_x(\{u_n\}, R)$. Therefore,  for all $R>0$,
\begin{equation}\label{int-hor}
E_x(\{u_n\}, R)=\bigcap_{\{v_n\}\in \mathcal M} E_x(\{v_n\}, R).
\end{equation}
In order to prove \eqref{intersI}, by \eqref{int-hor}, it is clear that $\hbox{II}_D^H([\{u_n\}])\subseteq\bigcap_{\{v_n\}\in \mathcal M} \hbox{II}_D^H([\{v_n\}])$. Conversely, let $q\in \bigcap_{\{v_n\}\in \mathcal M} \hbox{II}_D^H([\{v_n\}])$, let $R>0$ and let $z\in E_x(\{u_n\}, R)$. Then $z\in E_x(\{v_n\}, R)$ for all $\{v_n\}\in \mathcal M$ by \eqref{int-hor}. Since $E_x(\{v_n\}, R)$ is convex by Proposition \ref{horo-convex}, the segment $tz+(1-t)q$, $0<t\leq 1$ belongs to $E_x(\{v_n\}, R)$ for all $\{v_n\}\in \mathcal M$, thus it belongs to $E_x(\{u_n\}, R)$, proving that $q\in \hbox{II}_D^H([\{u_n\}])$.

In case $\{u_n\}$ converges to $p\in \partial D$, arguing similarly to \cite[Lemma 2.3]{AR}, one can prove that $p\in \hbox{II}_D^H([\{u_n\}])$. We give a sketch of the proof for the reader's convenience. Let $t\in (0,1)$ and define $f_t^n(w):=tw+(1-t)u_n$. Note that $f_t^n: D \to D$ is holomorphic, hence $f_t^n$ decreases the Kobayashi distance of $D$. Moreover,  $f_t^n(u_n)=u_n$ and $\lim_{n\to \infty} f_t^n(z)=tz+(1-t)p\in D$. Fix $R>0$. Let $z\in E_x(\{u_n\}, R)$. Then
\begin{equation*}
\begin{split}
&\limsup_{n\to \infty}[K_D(tz+(1-t)p, u_n)-K_D(x,u_n)]\\&\leq \limsup_{n\to \infty}[K_D(f_t^n(z), f_t^n(u_n))-K_D(x,u_n)]+\limsup_{n\to \infty}[K_D(tz+(1-t)p, f_t^n(z))]<\frac{1}{2}\log R.
\end{split}
\end{equation*} 
Therefore $tz+(1-t)p\in E_x(\{u_n\}, R)$, hence $p\in \hbox{II}(\{u_n\}, R)$ for all $R>0$ and thus $p\in \hbox{II}_D^H([\{u_n\}])$.
\end{proof}

In general, if $\{u_n\}\subset D$ is an admissible sequence which does not converge and $p\in \partial D$ is in the cluster set of $\{u_n\}$, it does not hold $p\in \hbox{II}_D^H([\{u_n\}])$, as the following example shows:

\begin{ex}
Let $D=\D^2$ be the bidisc. Consider the sequences $\{v_n\}$, $\{w_n\}$ given by $v_n=(0,1-1/n)$ and  $w_n=(1-1/n,0)$, for every $n\in \N, n\geq 1$. Since $K_{\D^2}((z_1,z_2), (w_1,w_2))=\max\{K_\D(z_1,w_1), K_\D(z_2,w_2)\}$, it follows easily that for all $R>0$
\[
E_{(0,0)}(\{v_n\}, R)=\D \times E_0^\D(1, R),
\]
where $E_0^\D(1, R)=\{\zeta\in \D: |1-\zeta|^2/(1-|\zeta|^2)<R\}$. Hence $\hbox{II}_D^H([\{v_n\}])=\overline{\D}\times \{1\}$. On the other hand,
\[
E_{(0,0)}(\{w_n\}, R)=E_0^\D(1, R) \times \D,
\]
and $\hbox{II}_D^H([\{w_n\}])=\{1\}\times \overline{\D}$. 

Let $\{u_n\}$ be the sequence defined by $u_{2n-1}=v_n, u_{2n}=w_n$, $n\geq 1$. Hence $\{u_n\}$ is admissible, since, for $R>0$ 
\[
E_{(0,0)}(\{u_n\}, R)=E_{(0,0)}(\{v_n\}, R)\cap E_{(0,0)}(\{w_n\}, R)=E_0^\D(1, R) \times E_0^\D(1, R).
\]
Then $\hbox{II}_D^H([\{u_n\}])=\{(1,1)\}$, and all points in the cluster set of $\{u_n\}$, namely $(1,0)$ and $(0,1)$, do not belong to $\hbox{II}_D^H([\{u_n\}])$.
\end{ex}

The previous example shows also that $\hbox{II}(\{u_n\}, R)$ is not convex in general. 

\subsection{Convex domains biholomorphic to strongly pseudoconvex domains} We examine now the case of convex domains (with no regularity assumption on the boundary) biholomorphic to strongly pseudoconvex domains. We  start with the following general result:

\begin{lem}\label{M-admiss}
Let $M$ be a complex manifold biholomorphic to a bounded strongly pseudoconvex domain with $C^3$ boundary, $x\in M$. Let $\{u_n\}$ be an admissible sequence in $M$. Let $\{z_n\}\subset M$ be a non relatively compact sequence which is eventually contained in $E_x(\{u_n\}, R)$ for some $R>0$. Then $\{z_n\}$ is admissible and it is equivalent to $\{u_n\}$. 
\end{lem}

\begin{proof}
Let $F: D\to M$ be a biholomorphism between a bounded strongly pseudoconvex domain with $C^3$ boundary and $M$. 

The sequence $\{F^{-1}(u_n)\}$ is admissible in $D$. By Proposition \ref{convergenceadmsp}, there exists $p\in \partial D$ such that $\{F^{-1}(u_n)\}$ converges to $p$. Since $\{z_n\}$ is eventually contained in $E^M_x(\{u_n\}, R)$, then $\{F^{-1}(z_n)\}$ is eventually contained in $E^D_{F^{-1}(x)}(\{F^{-1}(u_n)\}, R)$. Moreover, $\{F^{-1}(z_n)\}$ is not relatively compact in $D$, as $\{z_n\}$ is not in $M$. Therefore every limit of $\{F^{-1}(z_n)\}$ has to be contained in 
\begin{equation}\label{uno-in-p}
\overline{E^D_{F^{-1}(x)}(\{F^{-1}(u_n)\}, R)}\cap \partial D\subset \overline{F^D_{F^{-1}(x)}(p,R)}\cap \partial D=\{p\}
\end{equation}
by Theorem \ref{intersF}.(4). 

Therefore $\{F^{-1}(z_n)\}$ is converging to $p$. By Proposition \ref{stronglypseudo}, the sequence $\{F^{-1}(z_n)\}$ is admissible in $D$ and equivalent to $\{F^{-1}(u_n)\}$. Hence $\{z_n\}$ is admissible in $M$ and equivalent to $\{u_n\}$.
\end{proof}

As a corollary we have the following result

\begin{prop}\label{P:I-and-II}
Let $D\subset \C^N$ be a bounded convex domain. Assume $D$ is biholomorphic to a bounded strongly pseudoconvex domain with $C^3$ boundary. Let $\underline{x}\in \partial_H D$. Let $\{u_n\}$ be any admissible sequence in $D$ representing $\underline{x}$. Then for every $R>0$ it holds
\begin{equation}\label{I-and-II}
\hbox{II}(\{u_n\}, R)=\hbox{II}_D^H(\underline{x}).
\end{equation}
\end{prop}

\begin{proof}
By definition, $\hbox{II}_D^H(\underline{x})\subseteq \hbox{II}(\{u_n\}, R)$. Let now $p\in \hbox{II}(\{u_n\}, R)$. Fix $x\in D$. Then there exists a sequence $\{p_n\}\subset E_x(\{u_n\}, R)$ such that $\lim_{n\to \infty}p_n=p$. By Lemma \ref{M-admiss}, $\{p_n\}$ is an admissible sequence in $D$ and it is equivalent to $\{u_n\}$. By Lemma \ref{whois} it follows then that $p\in \hbox{II}_D^H(\underline{x})$. 
\end{proof}

The previous result allows us to generalize Lemma \ref{squeeze} to convex domains biholomorphic to strongly pseudoconvex domains:

\begin{lem}\label{squeezeconvex}
Let $D\subset \C^N$ be a bounded convex domain. Assume $D$ is biholomorphic to a bounded strongly pseudoconvex domain with $C^3$ boundary. Let $\{\underline{x}_j\}\subset \partial_H D$ be a sequence and let $V=\overline{\bigcup_j \hbox{II}_D^H(\underline{x}_j)}$. Then for every open neighborhood $U$ of $V$ there exists $R_0>0$ such that for every $0<R<R_0$ and for every admissible sequence $\{u_n^j\}$ representing $\underline{x}_j$ for some $j$, it holds
\[
E_x(\{u_n^j\}, R)\subset U.
\]
\end{lem}
\begin{proof}
We argue by contradiction. Let $\mathcal A_j$ denote the set of all admissible sequences in $D$ which represent $\underline{x}_j$.  If the result is not true, for every $n\in \N$ there exists  $z_n\in D$ such that 
\[
z_n\in \bigcup_j \bigcup_{\{u_m\}\in \mathcal A_j} E(\{u_m\}, \frac{1}{n})\cap (D\setminus U).
\]
 We can assume that $z_n\to z_0$ for some $z_0\in \overline{D}$. By Proposition \ref{outsideK}, $z_0\in \partial D$. In particular, given $R>1/n$, it holds $z_n\in \bigcup_j \bigcup_{\{u_m\}\in \mathcal A_j} E(\{u_m\}, R)$ for every $n\in \N$. Hence,
\[
z_0\in \overline{\bigcup_j \bigcup_{\{u_m\}\in \mathcal A_j} E(\{u_m\}, R)}\cap \partial D=V,
\]
by Proposition \ref{P:I-and-II}. Therefore, $z_0\in V\cap (\partial D\setminus U)=\emptyset$, a contradiction.
\end{proof}

As in the strongly pseudoconvex case, the previous lemma allows to relate Euclidean topology with horosphere topology: 

\begin{cor}\label{converge-bene}
Let $D\subset \C^N$ be a bounded convex domain. Assume $D$ is biholomorphic to a bounded strongly pseudoconvex domain with $C^3$ boundary. Also, assume that for every $\underline{x}\in \partial_H D$ there exists $p_{\underline{x}}\in \partial D$ such that  $\hbox{II}_D^H(\underline{x})=\{p_{\underline{x}}\}$. If $\{\underline{x}_j\}\subset \partial_H D$ is a sequence converging to $\underline{x}\in \partial_H D$ in the horosphere topology, then $\lim_{j\to \infty}p_{\underline{x}_j}=p_{\underline{x}}$ in the Euclidean topology. 
\end{cor}

\begin{proof}
Assume the conclusion of the corollary is not true and, possibly up to extracting subsequences, assume that $p_{\underline{x}_j}\to q$ for some $q\in \partial D\setminus\{p_{\underline{x}}\}$. Let $V=\overline{\cup_j \{p_{\underline{x}_j}\}}$. Let $U_0, U_1$ be two open sets in $\C^N$ such that $V\subset U_0$, $\{p_{\underline{x}}\}\subset U_1$ and $U_0\cap U_1=\emptyset$. By Lemma \ref{squeezeconvex}, there exists $R_0>0$ such that for all $0<R<R_0$ and all admissible sequences $\{u^j_n\}$ representing $\underline{x}_j$ it holds $E_x(\{u_n^j\}, R)\subset U_0$ and moreover, for all admissible sequences $\{u_n\}$ representing $\underline{x}$ it holds $E_x(\{u_n\}, R)\subset U_1$. But then, \eqref{def-horo-converge} can never be satisfied for $0<R<R_0$, and   $\underline{x}_j$ cannot converge to $\underline{x}$ in the horosphere topology. A contradiction.
\end{proof}

\begin{theo}\label{I-for-convex}
Let $D\subset \C^N$ be a bounded convex domain. Assume $D$ is biholomorphic to a bounded strongly pseudoconvex domain with $C^3$ boundary. Also, assume that for every $\underline{x}\in \partial_H D$ there exists $p_{\underline{x}}\in \partial D$ such that  $\hbox{II}_D^H(\underline{x})=\{p_{\underline{x}}\}$. Then for every $\underline{x}\in \partial_H D$ it holds
\[
\hbox{I}_D^H(\underline{x})=\hbox{II}_D^H(\underline{x})=\{p_{\underline{x}}\}.
\]
\end{theo}

\begin{proof}
We argue by contradiction. Let $\{w_n\}\subset D$ be a sequence converging to $\underline{x}$ in the horosphere topology. Assume that $\{w_n\}$ converges to $q\in \partial D$ in the Euclidean topology, with $q\neq  p_{\underline{x}}$. By definition of convergence in the horosphere topology, there exist admissible sequences $\{u_n^j\}$ and $\{u_n\}$ in $D$ such that $\{u_n\}$ represents $\underline{x}$, and for every $R>0$ there exists $m_R$ such that for every $m\geq m_R$ it holds $w_m\in E_x(\{u_n^m\}, R)$ and $E_x(\{u_n^m\}, R)\cap E_x(\{u_n\}, R)\neq\emptyset$. Let $\underline{x}_m:=[\{u_n^m\}]\in \partial_H D$. By Remark \ref{conv-in-like-out}, the sequence $\{\underline{x}_j\}$ converges in the horosphere topology to $\underline{x}$. Therefore, by Corollary \ref{converge-bene}, the sequence $\{p_{\underline{x}_j}\}$ converges to $p_{\underline{x}}$ in the Euclidean topology. 

Let $U_0$ and $U_1$ be two open sets in $\C^N$ such that $q\in U_0$, $p_{\underline{x}}\in U_1$ and $U_0\cap U_1=\emptyset$. Without loss of generality, we can assume that $\{w_n\}\subset U_0$ and $\{p_{\underline{x}_j}\}\subset U_1$. By Lemma \ref{squeezeconvex}, there exists $R_0>0$ such that $E_x(\{u_n^m\}, R)\subset U_1$ for all $0<R<R_0$ and all $m\in \N$.

Therefore, given $0<R<R_0$, for $m>m_R$, we have $w_m\in U_0\cap E_x(\{u_n^m\}, R)\subset U_0\cap U_1=\emptyset$, a contradiction.
\end{proof}

As shown by the previous result, it is important to see which bounded convex domains biholomorphic to bounded strongly pseudoconvex domains have the property that $\hbox{II}_D^H(\underline{x})$ is a point for every $\underline{x}\in \partial_H D$. We conjecture that this is always the case, but presently we are able to prove it for bounded strictly $\C$-linearly convex domains and in case of convex domains biholomorphic to strongly convex domains. In order to state the result, we need a definition:

\begin{defi}
Let $D \subset \C^N$ be a bounded convex domain, $p\in \partial D$. A {\sl complex supporting functional} at $p$ is a $\C$-linear map $\sigma:\C^N \to \C$ such that $\Re \sigma(z)<\Re \sigma (p)$ for all $z\in D$. A {\sl complex supporting hyperplane} for $D$ at $p$ is an affine complex hyperplane $L$ of the form $L=p+\ker \sigma=\{z\in \C^N: \sigma(z)=\sigma(p)\}$ where $\sigma$ is a complex supporting functional at $p$. Let $\mathcal L_p$ denote the set of all complex supporting hyperplanes at $p$. We set
\[
\hbox{Ch}(p)=\bigcap_{L\in \mathcal L_p}L\cap \overline{D}. 
\]
\end{defi}
Clearly, $\hbox{Ch}(p)$ is a closed convex set containing $p$. 

\begin{defi}
A bounded convex domain $D\subset \C^N$ is {\sl strictly $\C$-linearly convex} if for every $p\in \partial D$ it holds $\hbox{Ch}(p)=\{p\}$.
\end{defi}

\begin{prop}\label{strict-biholo-convex}
Let $D\subset \C^N$ be a bounded strictly $\C$-linearly convex domain. Assume $D$ is biholomorphic to a bounded strongly pseudoconvex domain with $C^3$ boundary. Then for every $\underline{x}\in \partial_H D$ there exists a unique $p\in \partial D$ such that \begin{equation}\label{st-II-one}
\hbox{I}_D^H(\underline{x})=\hbox{II}_D^H(\underline{x})=\{p\}.
\end{equation}
\end{prop}
\begin{proof}
If we prove that $\hbox{II}_D^H(\underline{x})=\{p\}$, then the result follows from Theorem \ref{I-for-convex}.

Let $\underline{x}\in \partial_H D$.  Suppose by contradiction that $p, q\in \hbox{II}_D^H(\underline{x})$.
By hypothesis, there exist $\Omega\subset \C^N$ a bounded strongly pseudoconvex domain with $C^3$ boundary and a biholomorphism $F:\Omega \to D$.

Let $\{u_n\}$ be an admissible sequence representing $\underline{x}$. For every $R>0$, $\hbox{II}(\{u_n\}, R)=\hbox{II}_D^H(\underline{x})$ by \eqref{I-and-II}. Therefore,   since $\hbox{II}_D^H(\underline{x})$  is convex,  the real segment $[p,q]$ joining $p$ and $q$ is contained in $\hbox{II}(\{u_n\}, R)$. 

 Let $v=p-q$ and let $L:=\C(p-q)+q$. There are two possibilities: either $L\cap D\neq \emptyset$ or $L\cap D=\emptyset$. 
 
In case  $L\cap D\neq\emptyset$, since $\cup_{R>0}E^D_x(\{u_n\}, R)=D$, there exists $R>0$ such that $E^D_x(\{u_n\}, R)\cap L\neq \emptyset$. By convexity of $E^D_x(\{u_n\}, R)$ (see Proposition \ref{horo-convex}), $\Delta:=L\cap E^D_x(\{u_n\}, R)$ is a convex domain in $L$ whose boundary contains the segment $[p,q]$.
Hence by the uniformization theorem, there exists a biholomorphism $\varphi:\D \to \Delta$ and, by the Schwarz reflection principle, there exists an arc $A\subset \partial \D$ such that $\varphi$ extends analytically on $A$ and $\varphi(A)=[p,q]$. Consider the map $F^{-1}|_\Delta:\Delta\to \Omega$. Since it is not constant, there exists a linear projection $\pi: \C^N \to \C$ such that $\pi \circ F^{-1}_\Delta:\Delta\to \C$ is not constant. Therefore, the map $g:=\pi\circ F^{-1}\circ \varphi:\D \to \C$ is a holomorphic bounded map. We claim that there exists $a\in \C$ such that
\begin{equation}\label{Fatou}
\lim_{r\to 1^-}g(r\zeta)=a \quad \forall \zeta\in A,
\end{equation}
which, by Fatou's lemma, implies that $g$ is constant, reaching a contradiction. 

In order to prove \eqref{Fatou},  we just note that  if $\{w_n\}\subset \Delta$ is a sequence converging to a point $\zeta\in [p,q]$, since $\{w_n\}\subset E^D_x(\{u_n\}, R)$, it follows that $\{F^{-1}(w_n)\}$ is a non-relatively compact sequence in $\Omega$ contained in the horosphere $E^{\Omega}_{F^{-1}(x)}(\{F^{-1}(u_n)\}, R)$ and thus there exists a point $u\in \partial \Omega$ (which depends only on $\{F^{-1}(u_n)\}$) such that $\{F^{-1}(w_n)\}$ converges to $u$ (see \eqref{uno-in-p}). From this, \eqref{Fatou} follows at once. 

Next, assume that $L\cap D=\emptyset$.  Then $[p,q]\subseteq L\cap \overline{D}=L\cap \partial D$. Let $\xi\in (p,q)$. Let $H\in \mathcal L_\xi$ be a complex supporting hyperplane for $D$ at $\xi$. If $\sigma:\C^N \to \C$ is a complex supporting functional such that $H=\{z\in \C^N: \sigma(z-\xi)=0\}$ then $\sigma(p-q)=0$, since $t(p-q)+\xi \in \partial D$ for $t\in \R$, $|t|<<1$.  This proves that $L\subset H$. But then,  $[p,q]\in \overline{D}\cap H$, and by the arbitrariness of $H$, it follows that $[p,q]\subset \hbox{Ch}(\xi)$, contradicting  $D$ being strictly $\C$-linearly convex. 

Therefore, $\hbox{II}_D^H(\underline{x})$ consists of one point.
\end{proof}

The hypothesis of $D$ being strictly $\C$-linearly convex  in the previous proposition is not necessary, as the following example shows:

\begin{ex}
Let $D:=\{(z_1,z_2)\in \C^2: \Re z_1>2(\Re z_2)^2\}$. It is easy to see that $D$ is convex, and it is biholomorphic to the Siegel domain $\mathbb H^2:=\{(w_1,w_2)\in \C^2: \Re w_1>|w_2|^2\}$ via the map $\mathbb H^2\ni (w_1,w_2)\mapsto (w_1+w_2^2, w_2)\in D$. The Siegel domain $\mathbb H^2$ is nothing but the unbounded realization of the ball $\B^2$ via the generalized Cayley transform $\B^2 \ni (z_1,z_2)\mapsto (\frac{1+z_1}{1-z_1}, \frac{z_2}{1-z_1}) \in \mathbb H^2$, hence, there exists a biholomorphism $F:\B^2\to D$ which extends as a homeomorphism on $\partial \B^2\setminus\{(1,0)\}$. In particular, $F^{-1}$ is continuous in a neighborhood of the point $(0,0)$, and $F^{-1}(0,0)=(-1,0)$. The sequence $\{w_n\}:=\{(\frac{1}{n},0)\}$ converges to $(0,0)$, and thus $\{F^{-1}(w_n)\}$ converges to $(-1,0)$ and it is admissible. It turns out that $\{w_n\}$ is admissible and  ${\hbox{II}}^D_H([\{u_n\}])=\{(-1,0)\}$. On the other hand, there is only one 
complex supporting hyperplane for $D$ at $(0,0)$, that is, $H=\{(z_1,z_2)\in \C^2: z_1=0\}$. Therefore
\[
\hbox{Ch}(0,0)=\partial D\cap H=\{(0,ti): t\in \R\}.
\]
\end{ex}

\subsection{Convex domains biholomorphic to strongly convex domains} 

The aim of this subsection is to prove Proposition \ref{strict-biholo-convex} for convex domains biholomorphic to strongly convex domains without any assumption on $\C$-strict linear convexity. 

We first recall some notions of the Gromov hyperbolicity theory.

\begin{defi}\label{geod-def}
Let $(D,d)$ be a metric space.

\begin{itemize}
 \item A curve $\gamma: [a,b] \rightarrow D$ is a {\it geodesic} if $\gamma$ is an isometry for the usual distance function on $[a,b] \subset \R$, i.e., $d(\gamma (t_1), \gamma (t_2))= \vert t_1-t_2 \vert$ for all $t_1,t_2 \in [a,b]$. We call $\gamma([a,b])$ a {\it geodesic segment}.
 \item The metric space $(D,d)$ is said to be a {\it geodesic metric space} if any two points in $D$ are connected by a geodesic.
 \item A {\it geodesic triangle} in $D$ is a union of images of three geodesics $\gamma_i: [a_i,b_i] \rightarrow D$, $i=1,2,3$,  such that $\gamma_i(b_i)=\gamma_{i+1}(a_{i+1})$ where the indices are taken modulo $3$.  The image of each $\gamma_i$ is called a {\it side} of the geodesic triangle.
 \item A geodesic metric space $(D,d)$ is
{\it Gromov hyperbolic} or $\delta$-hyperbolic if there exists $\delta >0$ such that for any geodesic triangle in $D$ the image of every side is contained in the $\delta$-neighborhood of union of the other two sides.
 \item Let $A \geq 1$ and $B > 0$. We say that $\gamma : [a,b] \rightarrow D$ is a {\sl $(A,B)$ quasi-geodesic} if for every $t,t' \in ]a,b[$ we have:
 $$
 \frac{1}{A}|t-t'| - B \leq d(\gamma(t),\gamma(t')) \leq A |t-t'| + B.
 $$
\end{itemize}
\end{defi}

Also, we need to prove some preliminary lemmas:

\begin{lem}\label{no-disc}
Let $D\subset \C^N$ be a  convex domain biholomorphic to a bounded strongly pseudoconvex domain. Let $\{p_j\}$ and $\{q_j\}$ be two sequences in $D$ which converge to two different boundary points. Then
\begin{equation}\label{go-infty}
\lim_{j\to \infty}K_D(p_j, q_j)=\infty.
\end{equation} 
\end{lem}
\begin{proof}
Assume \eqref{go-infty} is not true. Then, up to subsequences, we can assume that there exists $C>0$ such that for all $j$
\begin{equation}\label{bound-K-C}
K_D(p_j,q_j)\leq C.
\end{equation} 
Since $D$ is a complete hyperbolic convex domain, for every $j$ there exists a complex geodesic $\varphi_j:\D \to D$ such that $\varphi_j(0)=p_j$ and $\varphi_j(t_j)=q_j$  for some $t_j\in (0,1)$, see  \cite[Lemma 3.3]{BS}. By \eqref{bound-K-C} it follows
\[
K_\D(0,t_j)=K_D(\varphi_j(p_j),\varphi_j(q_j))=K_D(p_j,q_j)\leq C,
\]
hence, there exists $c\in (0,1)$ such that $t_j<c$ for all $j$, and we can assume without loss of generality that $t_j\to t_0$ for some $t_0\in [0,1)$. 
Being $D$ taut, up to subsequences, we can also assume that $\{\varphi_j\}$ converges uniformly on compacta to a holomorphic map $\varphi: \D \to \overline{D}$.  Moreover, 
\[
\varphi(t_0)=\lim_{j\to \infty}\varphi_j(t_j)=\lim_{j\to\infty}q_j=q,
\]
and similarly, $\varphi(0)=p$. This implies that $t_0\neq 0$ and $\varphi$ is not constant. In particular, $\partial D$ contains (non-constant) analytic discs. But $D$ is Gromov hyperbolic (with respect to the Kobayashi distance) since it is biholomorphic to a bounded strongly pseudoconvex domain (see \cite{BB}) and by \cite[Thm. 3.1]{Zi}, $\partial D$ cannot contain (non-constant) analytic discs, a contradiction.
\end{proof}

\begin{lem}\label{quasi-lem}
 Let $D$ be a hyperbolic  convex domain in $\mathbb C^n$. Let $x \in D$ and let $p \in \partial D$. There exists a bounded open neighborhood $U$ of $p$ and there exist $A>1, B > 0$ such that for every sequence $\{z_j\}$ of points in $D\cap U$ converging to $p$, the line segment $[x,z_j]$ is a $(A,B)$ quasi-geodesic.
\end{lem}

\begin{proof} 
The proof of Lemma~\ref{quasi-lem} is a modification of the proof of \cite[Proposition 4.2]{GS}. For convenience of the reader we present it, adapted to our situation.

For $q \in \mathbb C^n$, $\nu \in \mathbb C^n$ with $|\nu|=1$ and $\alpha \in (0,\pi)$, let $C^+(q,\nu,\alpha)$ denote the half cone of vertex $q$, direction $\mathbb R^+ \nu$, and aperture $\alpha$.

Let $\nu = \frac{x-p}{|x-p|}$. Since the real  segment $[x,p)$ is contained in $D$, there exists a bounded open neighborhood $V$ of $p$ and there exists $s_0 > 0$ such that for every $q \in V \cap \partial D$ the segment $[q+s_0\nu,q)$ is contained in $D$.  Let $\alpha > 0$ be such that $C^+(p,\nu,\alpha) \cap B(p,s_0)$ is contained in $D$. Shrinking $V$ and changing $s_0$ if necessary,  we can assume that $x\not\in V$ and that the intersection $C^+(q,\nu,\alpha) \cap B(q,s_0)$ is contained in $D$ for every $q \in \partial D \cap V$. Still shrinking $V$ if necessary, we may also assume that for every $q \in \partial D \cap V$ the real segment $[x,q)$ is contained in $C^+(q,\nu,\alpha/2)$ and that the set $S_V:=\bigcup_{q \in \partial D \cap V}\left(C^+(q,\nu,\alpha) \cap \partial B(q,s_0)\right)$ is relatively compact in $D$. Let $W(x, S_V)$ be the convex hull of $x$ and $S_V$. Note that $W(x, S_V)$ is relatively compact in $D$. Let
\begin{equation}\label{eqBcon}
B_1:=\sup_{w_0,w_1\in W(x, S_V)} K_D(w_0,w_1)<+\infty.
\end{equation}
Finally, let $U\subset V$ be an open neighborhood of $p$ with the property that for every $z\in U\cap D$, there exists $q\in V\cap \partial D$ such that $z\in C^+(q,\nu,\alpha/2)\cap B(q,s_0)$.

If $\gamma:[a,b]\to D$ is a piecewise $C^1$  smooth curve, we denote by $l^K_D(\gamma([a,b]))$ its Kobayashi length, {\sl i.e.}, $l^K_D(\gamma([a,b]))=\int_a^b k_D(\gamma(t);\gamma'(t))dt$.

Now, let $\{z_j\}\subset D\cap U$ converging to $p$. For every $j \in \N$, we parametrize the real segment $[x,z_j]$ with respect to Kobayashi arc length, meaning that we consider a piecewise $\mathcal C^1$ curve $\gamma_j : [0,T_j] \rightarrow D$ such that $\gamma_j([0,T_j]) = [x,z_j]$ and the Kobayashi length $l^K_D(\gamma_j([a,b]))=|b-a|$ for all $0\leq a\leq b\leq T_j$, for every $j\in \N$.

By construction, for every $j\in \N$ there exists $q_j\in V\cap \partial D$ such that $z_j\in C^+(q_j,\nu,\alpha/2)\cap B(q_j,s_0)$ and $[x,q_j)\subset C^+(q_j,\nu,\alpha/2)$. Since $x\not\in U$ by definition, for every $j\in \N$ there exists a unique point $w_j\in \partial B(q_j,s_0)\cap [x,z_j)$. Let $R_j\in (0, T_j)$ be such that $\gamma_j(R_j)=w_j$. 

\vskip 0,1cm
\noindent{\sl Claim}. There exist $A>1$ and $B_2>0$  such that for every $q\in V\cap \partial D$, given any $w_0\in  C^+(q,\nu,\alpha/2)\cap \partial B(q,s_0)$  the real segment $[w_0,q)$ is a $(A,B_2)$ quasi-geodesic.

\vskip 0,1cm

Assuming the claim for the moment, the proof ends as follows. Let $B:=B_1+B_2$, where $B_1$ is given by \eqref{eqBcon}.

\noindent{\sl Case 1.} If $0\leq s\leq t\leq R_j$ then 
\[
l_D^K(\gamma_j([s,t]))\leq B_1\leq A K_D(\gamma_j(s),\gamma_j(t))+B.
\]

\noindent{\sl Case 2.} If $0\leq s\leq  R_j\leq t\leq T_j$ then by the Claim, 
\begin{equation}\label{eq-inc2}
l_D^K(\gamma_j([s,t]))= l_D^K(\gamma_j([s,R_j]))+ l_D^K(\gamma_j([R_j,t]))\leq B_1+AK_D(\gamma(R_j), \gamma(t))+B_2.
\end{equation}
Now, $\gamma(R_j)=w_j=r_j \gamma(s)+(1-r_j)\gamma(t)$ for some $r_j\in (0,1)$. 
Since  $K_D$ is a  convex function (see the proof of Prop. \ref{horo-convex}), then 
\begin{equation*}
\begin{split}
K_D(\gamma(R_j), \gamma(t))&=K_D(r_j \gamma(s)+(1-r_j)\gamma(t), \gamma(t))\\ &\leq \max\{K_D(\gamma(s),\gamma(t)), K_D(\gamma(t),\gamma(t))\}=K_D(\gamma(s),\gamma(t)).
\end{split}
\end{equation*}
Hence, by \eqref{eq-inc2} we have
\[
l_D^K(\gamma_j([s,t]))\leq AK_D(\gamma(s), \gamma(t))+B.
\]

\noindent{\sl Case 3.} If $  R_j\leq s\leq t\leq T_j$ then $\gamma(s)$ and $\gamma(t)$ belong to the real segment $[w_j, q_j]$ which, by the Claim,  is a $(A,B_2)$ quasi-geodesic, and, hence, a $(A,B)$ quasi-geodesic.

Finally, since for every $s,t\in [0,T_j]$ we have
\[
\frac{1}{A} K_D(\gamma_j(s),\gamma_j(t))-B\leq K_D(\gamma_j(s),\gamma_j(t))\leq l_D^K(\gamma_j([s,t])),
\]
the previous arguments show that $[x,z_j]$ are $(A,B)$-quasi geodesics for every $j$.

We are left to prove the Claim. 
For a point $y \in D$, as usual, we denote by $\delta(y)$ the Euclidean distance from $y$ to $\partial D$ and by $\delta(y, \partial D_{\nu,y})$ the Euclidean distance from $y$ to $\partial D$ along the complex line $\mathbb C \nu$, {\sl i.e.}, $\delta(y, \partial D_{\nu,y}):=\hbox{dist}(y, \partial D\cap(\C\nu+y))$. We recall the standard estimates on the Kobayashi infinitesimal metric  on convex domains (see, {\sl e.g.}, \cite{BP}):
\begin{equation}\label{StandardEst}
\frac{|v|}{2\delta(z,\partial D_{v,z})}\leq k_D(z;v)\leq \frac{|v|}{\delta(z,\partial D_{v,z})}.
\end{equation}
Let now $q\in V\cap \partial D$ and $w_0\in  C^+(q,\nu,\alpha/2)\cap \partial B(q,s_0)$. Let $\eta(r):=(1-r)w_0+rq$, $r\in [0,1)$. Also, for a point $z\in C^+(q,\nu,\alpha/2)\cap \overline{B(q,s_0)}$, we let $\tilde{z}\in [q+s_0\nu, q]$ be the (real) orthogonal projection of $z$ on the axis of $C^+(q,\nu,\alpha/2)$. Let  $\nu_q:=\frac{w_0-q}{|w_0-q|}$. Since $C^+(q,\nu,\alpha)\cap B(q,s_0)\subset D$, there exists a constant $C>0$ (depending only on $s_0$ and $\alpha$ but not on $q$ and $w_0\in C^+(q,\nu,\alpha/2)\cap \partial B(q,s_0)$ ) such that, for every $z\in [w_0,q)$ it holds
\[
\delta(\tilde z,\partial D_{\nu,\tilde z})\leq C\delta(z,\partial D_{\nu_q,z}).
\]

For every $r \in [0,1)$, we have $\widetilde{\eta(r)} = \Re\langle \eta(r)-q ,\nu\rangle\nu+q=(1-r)\Re\langle w_0-q,\nu\rangle \nu+q$. Note that $\Re\langle w_0-q ,\nu\rangle>0$.
Hence, for every $0\leq r_1\leq r_2<1$, we have by~\eqref{StandardEst}:
$$
\begin{array}{lllll}
l^K_D(\eta[r_1,r_2]) & = & \displaystyle \int_{r_1}^{r_2}k_D(\eta(r);\eta'(r))dr & \leq & \displaystyle \int_{r_1}^{r_2}\frac{|q-w_0|}{\delta(\eta(r),\partial D_{\nu_q, \eta(r)})}dr\\
& & & & \\
& & & \leq & \displaystyle \frac{1}{C}\frac{|q-w_0|}{|\Re\langle w_0-q ,\nu\rangle|}\int_{r_1}^{r_2}\frac{|\Re\langle w_0-q ,\nu\rangle|}{\delta(\widetilde{\eta(r)},\partial D_{\nu, \widetilde{\eta(r)}})}dr\\
& & & & \\
& & &  \leq & \displaystyle \frac{1}{C}\frac{|q-w_0|}{|\Re\langle w_0-q ,\nu\rangle |}2\int_{r_1}^{r_2}k_D(\widetilde{\eta(r)};\widetilde{\eta'(r)})dr\\
& & & & \\
& & & = & \displaystyle \frac{2|q-w_0|}{C|\Re\langle w_0-q ,\nu\rangle |}l^K_D(\widetilde{\eta}[r_1,r_2]).
\end{array}
$$
Therefore,
\[
l^K_D(\eta[r_1,r_2])\leq \frac{2|q-w_0|}{C|\Re\langle w_0-q ,\nu\rangle |}l^K_D(\widetilde{\eta}[r_1,r_2]).
\]
Hence, the Claim is equivalent to the following
\vskip 0,1cm
\noindent{\sl Claim'}. There exist $A>1$ and $B_2>0$  such that for every $q\in V\cap \partial D$,  the real segment $[q+s_0\nu,q)$ is a $(A,B_2)$ quasi-geodesic.

\vskip 0,1cm

Let $\gamma : \mathbb R^+ \rightarrow  D$ parametrizing the real segment $[q+s_0\nu,q)$ by the Kobayashi arc length, with $\gamma(0)=q+s_0 \nu$.

Since $D$ is convex, and since $C^+(q,\nu,\alpha) \cap B(q,s_0) \subset D$, there exists $A'_{\alpha}>1$ such that for every $s \in \mathbb R^+$:
$$
\delta(\gamma(s), \partial D_{\nu, \gamma(s)}) \geq \delta(\gamma(s)) \geq \frac{1}{A'_{\alpha}} |\gamma(s)-q|.
$$
Let $H_q$ be a real half-space containing $D$ and such that $\partial H_q$ is a real supporting hyperplane for $D$ at $q$. Hence, we have for every $0 < t < t'$ (see below for explanation of the various inequalities):
$$
\begin{array}{lllll}
K_{D}(\gamma(t),\gamma(t')) & \leq & l^K_D([\gamma(t),\gamma(t')]) & \leq & \displaystyle \int_t^{t'}\displaystyle \frac{|\gamma'(s)|}{\delta(\gamma(s),\partial D_{\nu, \gamma(s)})}ds\\
 & & & & \\
 & & & \leq & A'_{\alpha} \displaystyle \int_t^{t'}\displaystyle \frac{|\gamma'(s)|}{|\gamma(s)-q|}ds\\
 & & & & \\
 & & & \leq & A'_{\alpha} \displaystyle \int_t^{t'}\displaystyle \frac{|\gamma'(s)|}{\delta(\gamma(s),\partial (H_q)_{\nu, \gamma(s)})}ds\\
 & & & & \\
 & & & \leq & 2A'_{\alpha} \ l^K_{H_q}([\gamma(t),\gamma(t')])\\
 & & & & \\
 & & & \leq & A_{\alpha} K_{H_q}(\gamma(t),\gamma(t')) + B_{\alpha}\\
 & & & & \\
 & & & \leq & A_{\alpha} K_{D}(\gamma(t),\gamma(t')) + B_{\alpha},
\end{array} 
$$
for some positive constants $A_{\alpha} > 1$ and $B_{\alpha}>0$, depending only on $\alpha$.

The second and the fifth inequalities follow from \eqref{StandardEst}. The fourth inequality uses the fact that $\delta(\gamma(s),\partial (H_q)_{\nu, \gamma(s)}) \leq |\gamma(s)-q|$ since $q \in \partial H_q \cap (\mathbb C \nu+\gamma(s))$ for all $s$. 
The sixth inequality uses the fact that every real segment in $H_q$ is a quasi-geodesic on the half space $H_q$ with constants depending only on the angle between the segment and $\partial H_q$.
Finally, the last inequality uses the fact that $D$ is contained in $H_q$.

This proves Claim' and the proof is completed.
\end{proof}

We will need the following statement (see also Lemma~3.3. in \cite{GS2}).
\begin{lem}\label{unique-geodesic}
Let $D\subset \C^N$ be a  convex domain biholomorphic to a bounded strongly convex domain with $C^3$ boundary. Then for every couple of points $z_0, w_0\in D$ there exists a unique real geodesic  for the Kobayashi distance which joins $z_0$ and $w_0$. 
\end{lem}

\begin{proof}
Since $D$ is biholomorphic to a bounded strongly convex domain, Lempert's theory \cite{L, Le1,Le2} (see also \cite{Aba}) implies that for every $z_0,w_0\in D$, $z_0\neq w_0$, there exists a unique complex geodesic whose image contains $z_0, w_0$. In other words, there exists $\varphi:\D \to D$ holomorphic such that $K_\D(\zeta,\eta)=K_D(\varphi(\zeta), \varphi(\eta))$ for all $\zeta, \eta\in \D$ and $z_0, w_0\in \varphi(\D)$ and,  moreover, if $\tilde{\varphi}:\D \to D$ is any holomorphic map such that there exist $\zeta_0, \zeta_1\in \D$ with $\tilde{\varphi}(\zeta_0)=z_0$, $\tilde{\varphi}(\zeta_1)=w_0$ and $K_D(z_0,w_0)=K_\D(\zeta_0,\zeta_1)$, then there exists an automorphism $\theta:\D \to \D$ such that $\tilde{\varphi}=\varphi\circ \theta$. If $\varphi: \D\to D$ is a complex geodesic, there exists a  holomorphic retraction $\rho: D\to D$, called the Lempert projection, such that $\rho\circ \rho=\rho$, $\rho(D)=\varphi(\D)$. In what follows we will use the following fact: if $z_0\in \varphi(\D)$ and $w\in D\setminus \varphi(\D)$ then
\begin{equation}\label{strict-ineq-convex-stric}
K_D(z_0, \rho(w))=K_D(\rho(z_0), \rho(w))<K_D(z_0,w).
\end{equation}
In order to prove inequality \eqref{strict-ineq-convex-stric} we can assume that $D$ is a bounded strongly convex domain with $C^3$ boundary. Let $\varphi:\D \to D$ be a complex geodesic and let $\rho:D\to \varphi(\D)$ be the associated Lempert projection. Fix $R>0$ and $z_0\in \varphi(\D)$. Since $D$ is assumed to be strongly convex, the Kobayashi ball $B_K(z_0, R)$ of center $z_0$ and radius $R>0$ is strongly convex as well, and  the Lempert projection $\rho$ has affine fibers (see \cite[Prop. 3.3]{BPT}). Taking into account that $\rho$ contracts the Kobayashi distance $K_D$, it follows that for every $\zeta\in \D$ with $K_D(\varphi(\zeta), z_0)=R$ it holds $\rho^{-1}(\varphi(\zeta))\cap \overline{B_K(z_0, R)}=\{\varphi(\zeta)\}$. From this,   \eqref{strict-ineq-convex-stric} follows at once by considering $\zeta \in \mathbb D$ with $\varphi(\zeta) = \rho(w)$.

Now, let $z_0, w_0\in D$. Hence, there exists a unique complex geodesic $\varphi:\D \to D$ such that $z_0, w_0\in \varphi(\D)$. Since $\varphi$ is an isometry between  $K_\D$ and $K_D$,  it follows that there exists a real geodesic for $K_D$, call it $\gamma:[0,a]\to \varphi(\D)$, $\gamma(0)=z_0, \gamma(a)=w_0$ for  $a=k_D(z_0,w_0)$,  which is contained in $\varphi(\D)$. Moreover, $\gamma$ is the only real geodesic joining $z_0$ and $w_0$ contained in $\varphi(\D)$. 

Let assume that $\tilde{\gamma}:[0, a]\to D$,  is another real geodesic for $K_D$ such that $\tilde{\gamma}(0)=z_0$ and $\tilde{\gamma}(a)=w_0$ and $\tilde{\gamma}([0, a])\neq \gamma([0,a])$. Then the image of $\tilde{\gamma}$ is not contained in $\varphi(\D)$. Hence, there exists $t\in (0,a)$ such that $z_1:=\tilde{\gamma}(t)\not\in \varphi(\D)$. Then, since $\tilde{\gamma}$ is a real geodesic for the Kobayashi distance,
\[
K_D(z_0,w_0)=K_D(z_0,z_1)+K_D(z_1,w_0).
\]
Now, let $\rho:D\to \varphi(\D)$ be the Lempert projection associated with $\varphi$. Since $\rho$ contracts the Kobayashi distance, by \eqref{strict-ineq-convex-stric} it follows:
\begin{equation*}
\begin{split}
K_D(z_0,w_0)&=K_D(z_0, z_1)+K_D(z_1,w_0)> k_D(\rho(z_0), \rho(z_1))+k_D(\rho(z_1), \rho(w_0))\\&\geq K_D(\rho(z_0), \rho(w_0))=K_D(z_0,w_0),
\end{split}
\end{equation*}
a contradiction, and the result is proved.
\end{proof}

\begin{theo}\label{ball-biholo-convex}
Let $D\subset \C^N$ be a bounded convex domain biholomorphic to a bounded strongly convex domain with $C^3$ boundary. Then for every $\underline{x}\in \partial_H D$ there exists a unique $p\in \partial D$ such that 
\begin{equation}\label{ball-II-one}
\hbox{I}_D^H(\underline{x})=\hbox{II}_D^H(\underline{x})=\{p\}.
\end{equation}
\end{theo}

\begin{proof} If we prove that $\hbox{II}_D^H(\underline{x})=\{p\}$, then the result follows from Theorem \ref{I-for-convex}.

Assume by contradiction that there exists $q\in \hbox{II}_D^H(\underline{x})$ with $p\neq q$. Fix $x\in D$.

Since $\hbox{II}_D^H(\underline{x})$ is convex, the real segment $[p,q]\subset \hbox{II}_D^H(\underline{x})$. Let $\{u_n\}$ be an admissible sequence which represents $\underline{x}$. By \eqref{I-and-II}, the segment $[p,q]\subset \hbox{II}(\{u_n\}, R)$ for all $R>0$. In particular, if we take $R>1$ so that $x\in E_x(\{u_n\}, R)$, the real  segments $[x,p)$ and $[x,q)$ are contained in $E_x(\{u_n\}, R)$ (by convexity, see Proposition \ref{horo-convex}). Set $p(t):=(1-t)p+t x$ and  $q(t):=(1-t)q+tx$, $t\in [0,1)$. 
Then for every $t\in [0,1)$,  $p(t), q(t)\in E_x(\{u_n\}, R)$.

Let $\{t_j\}\subset (0,1)$ be any sequence which converges to $1$. By Lemma \ref{M-admiss},  the sequences $\{p(t_j)\}, \{q(t_j)\}$ are admissible and equivalent to $\{u_n\}$.
 
By hypothesis, there exist $\Omega\subset \C^N$ a bounded strongly convex domain with $C^3$ boundary and a biholomorphism $F:\Omega \to D$.

Then the sequences $\{F^{-1}(p(t_j))\}$ and $\{F^{-1}(q(t_j))\}$ belong to the same horosphere $E_{F^{-1}(x)}(\{F^{-1}(u_n)\}, R)$ and are equivalent. Also, both sequences are equivalent to $\{F^{-1}(u_n)\}$.
Hence, it follows from Proposition~\ref{convergenceadmsp} that the sequences $\{F^{-1}(p(t_j))\}$, $\{F^{-1}(q(t_j))\}$ and $\{F^{-1}(u_j)\}$ converge to the same boundary point $\xi \in \partial \Omega$. By the arbitrariness of $\{t_j\}$, it follows that, in fact, $\lim_{t\to 1}F^{-1}(p(t))=\lim_{t\to 1}F^{-1}(q(t))=\xi$. 

For a fixed $t\in (0,1)$, let $\gamma^p_t:[0, R_t]\to D$ ({\sl respectively} $\gamma^q_t:[0, R'_t]\to D$) be the  real (Kobayashi) geodesic in $D$ such that $\gamma^p_t(0)=x$ and $\gamma^p_t(R_t)=p(t)$ ({\sl respect.}, $\gamma^q_t(0)=x$ and $\gamma^q_t(R'_t)=q(t)$). 

Then,  by Lemma \ref{unique-geodesic}, $F^{-1}\circ \gamma_t^q:[0,R'_t]\to \Omega$ is the unique real  geodesic joining $F^{-1}(x)$ and $F^{-1}(q(t))$, while $F^{-1}\circ \gamma_t^p:[0,R_t]\to \Omega$ is the unique real geodesic joining $F^{-1}(x)$ and $F^{-1}(p(t))$. By Lempert's theory (see, {\sl e.g.}, \cite{Aba, BPT}), since $\Omega$ is a $C^3$ strongly bounded convex domain, there exists a unique real geodesic  $\tilde\gamma:[0,\infty)\to \Omega$ such that $\tilde\gamma(0)=F^{-1}(x)$ and $\lim_{t\to\infty}\tilde\gamma(t)=\xi$ and, since $F^{-1}(p(t))$ and $F^{-1}(q(t))$ converge to $\xi$ as $t\to 1$, the Kobayashi geodesics $F^{-1}(\gamma^p_t)$ and $F^{-1}(\gamma^q_t)$ converge uniformly on compacta of $[0,\infty)$ to $\tilde\gamma$. 

Let $\gamma:=F\circ \tilde\gamma$. Since $F$ is an isometry for the Kobayashi distance, this implies that for every $\epsilon>0$ and for every $R>0$ there exists $t_0\in (0,1)$ such that for all $s\in [0,R]$ and $t\in [t_0,1)$ it holds 
\begin{equation}\label{eqi1}
K_D(\gamma^q_t(s), \gamma(s))<\epsilon, \quad K_D(\gamma^p_t(s), \gamma(s))<\epsilon.
\end{equation}
 
Now, by Lemma \ref{quasi-lem},  there exist $A,B>0$ such that for all $t\in (0,1)$, the segments $[x,p(t)]$ and $[x,q(t)]$ are $(A,B)$-quasi-geodesics.  Since $(D,K_D)$ is Gromov hyperbolic by \cite[Theorem 1.4]{BB}, by Gromov's shadowing lemma (see \cite[Th\'eor\`eme 11 p. 86]{GH}), for every $A,B>0$ fixed, there exists $M>0$ such that for every $(A,B)$-quasi-geodesic $\eta$ there exists a (real) geodesic $\hat{\eta}$ for $K_D$ such that $\eta$ belongs to the hyperbolic neighborhood $\mathcal N_M(\hat{\eta}):=\{w\in D: \exists z\in \hat{\eta}, K_D(w, z)<\epsilon\}$. Therefore, for every $t\in (0,1)$, we have
\begin{equation}\label{equi2}
\gamma^p_t([0, R_t])\subset  \mathcal N_M([0,p(t)])\subset \mathcal N_M([0,p)), \quad \gamma^q_t([0, R'_t])\subset  \mathcal N_M([0,q(t)])\subset \mathcal N_M([0,q)).
\end{equation}
Now, let fix $s\in [0,\infty)$. By \eqref{eqi1}, for every $t\in (0,1)$ sufficiently close to $1$, we have $K_D(\gamma(s), \gamma_t^p(s))<M$, while, by \eqref{equi2}, we have $\gamma_t^p(s)\in  \mathcal N_M([0,p))$. Hence, by the triangle inequality, $\gamma(s)\in \mathcal N_{2M}([0,p))$. Similarly, arguing with $\gamma_t^q(s)$, we see that $\gamma(s)\in \mathcal N_{2M}([0,q))$. Therefore, for all $s\in [0,\infty)$ it holds
\[
\gamma(s)\in U:=\mathcal N_{2M}([0,p))\cap \mathcal N_{2M}([0,q)).
\]
However, we claim that $U$ is relatively compact in $D$. If this is the case, we clearly  obtain a contradiction because $\gamma([0,\infty))$ is not relatively compact in $D$. 

Suppose by contradiction that $U$ is not relatively compact in $D$. Hence there exists a sequence $\{w_j\}\subset U$ converging to the boundary of $D$. By definition, this means that there exist two sequences $\{t_j\}, \{t_j'\}\subset [0,1)$ such that $K_D(w_j, p(t_j))<2M$ and $K_D(w_j, q(t'_j))<2M$ for every $j\in \N$. Since $\{w_j\}$ converges to the boundary of $D$ and $D$ is complete hyperbolic, it follows that $t_j\to 1$ and $t_j'\to 1$ as $j\to \infty$, that is, $q(t'_j)\to q$ and $p(t_j)\to p$ as $j\to \infty$. But, by the triangle inequality, for every $j\in \N$ it holds $K_D(p(t_j), q(t_j'))\leq 4M$, contradicting Lemma  \ref{no-disc}. \end{proof}

\section{Horosphere boundary versus Gromov boundary}\label{Gromov}

Different types of topological boundaries may be defined in the general context of metric spaces. For instance the construction of the Gromov boundary, using geodesic rays or sequences, is based on the following

\begin{defi}\label{gro-equiv}
 Let $(X,d)$ be a  metric space and let $w \in X$ be a base point.
 
 \begin{itemize}
  \item [(i)] Given $x,y \in X$, the Gromov product of $x$ and $y$ with respect to $w$ is defined by
  $(x,y)_w:=\frac{1}{2}\left(d(x,w) + d(y,w) - d(x,y)\right).
  $
  \item [(ii)] A geodesic ray $\gamma:[0,+\infty)\to X$, with $\gamma(0)=w$, is an isometry from $([0,+\infty),| \cdot |)$ to $(X,d)$  where $| \cdot |$ denotes the absolute value on $\mathbb R$; namely if $l(\gamma([s,t]))$ denotes the length of the curve $\gamma([s,t])$, we have:
  $$
  \forall s,t \geq 0,\ l(\gamma([s,t])) = d(\gamma(s),\gamma(t)) = |t-s|.
  $$
  \item [(iii)] Two geodesic rays $\gamma$ and $\tilde \gamma$ are equivalent (we write $\gamma \sim_r \tilde \gamma$) if they are at finite Hausdorff distance from each other, {\sl i.e.}, if there exists $\exists C > 0$ such that for all $t \geq 0$, it holds $d(\gamma(t), \tilde\gamma(t)) \leq C$. The relation $\sim_r$ is an equivalence relation on the set of geodesic rays.
  
  \item [(iv)] A sequence $\{x^\nu\}_\nu$ of points in $X$ tends to infinity if $\lim_{\nu,\mu \rightarrow \infty}(x^\nu,x^\mu)_w = +\infty$.
  
  \item [(v)] Two sequences $\{x^\nu\}_\nu$ and $\{y^\nu\}_\nu$ of points in $X$ are equivalent (we write $\{x^\nu\} \sim_s \{y^\nu\}$) if $\lim_{\nu \rightarrow \infty}(x^\nu,y^\nu)_w = + \infty$.
  \end{itemize}
  \end{defi}
  
There are different ways to define the Gromov boundaries, using geodesic rays, or using sequences. Both are equivalent when $X$ is a geodesic, proper, hyperbolic (in the sense of Gromov) metric space.

\begin{defi}\label{gro-top}

Let $(X,d)$ be a geodesic, proper, metric space.
\begin{itemize}
\item[(i)] The Gromov boundary $\partial_G^r X$ (with respect to rays) is defined by
  $$
  \partial_G^r X :=E^r/ \sim_r
  $$
  where $E^r$ denotes the set of geodesic rays.
\item [(ii)] Let $p \in \partial^r_G X$ and let $r> 0$. We consider
$$
V(q,r):=\{q \in \partial_G^r X:\ \exists \gamma \in p,\ \exists \tilde \gamma \in q,\ \liminf_{s,t \rightarrow \infty}(\gamma(t),\tilde \gamma(s))_w \geq r\}. 
$$
Then $(V(p,r))_{p \in \partial_G^r X, r > 0}$ is a basis of open neighborhoods for the Gromov topology $\mathcal T_G^r(\partial_G^r X)$ on $\partial_G^r X$.
 \end{itemize}

\end{defi}

\begin{rem}
In the previous definitions it is not necessary that the domain be (Gromov) hyperbolic. However, if $(X,d)$ is not (Gromov) hyperbolic then the Gromov boundary $\partial_G^r(X)$ may fail to be Hausdorff. 
\end{rem}

One may also define the Gromov boundary using equivalent sequences:

\begin{defi}\label{gor-seq}
Let $(X,d)$ be a proper, hyperbolic (in the sense of Gromov) metric space.
 The Gromov boundary $\partial_G^s X$ (with respect to sequences) is defined by
 $$
 \partial_G^s X:=E^s/\sim_s
 $$
where $E^s$ denotes the set of sequences converging to infinity.

\end{defi}

\begin{rem}

\begin{itemize}
\item In case $(X,d)$ is proper, geodesic, (Gromov) hyperbolic, there exists a bijection between $\partial_G^r X$ and $\partial_G^s X$.
 \item If $(X,d)$ is not (Gromov) hyperbolic then $\sim_s$ may not be an equivalence relation. This is the case for the bidisc endowed with its Kobayashi distance $K_{\mathbb D^2}$: $\{(1-1/n,0)\} \sim_s \{(0,1-1/n)\}$, $\{(0,1-1/n)\} \sim_s \{(-1+1/n,0)\}$ but $\{(1-1/n,0)\} \not\sim_s \{(-1+1/n,0)\}$. In particular we can not define $\partial_G^s \mathbb D^2$.
\end{itemize}
\end{rem}

It is natural to study the relations between the Gromov boundary $\partial_G^r M$ and the horosphere boundary $\partial_H M$ for a complete (Kobayashi) hyperbolic manifold $(M,K_M)$.
Indeed, since $(M,K_M)$ is a length space, it follows from the Hopf-Rinow Theorem that any two points can be joined by a geodesic segment.

Let $\mathbb D^2:=\{(z_1,z_2) \in \mathbb C^n :\ |z_1|<1,|z_2|<1\}$ be the bidisc in $\mathbb C^n$. The remaining part of this section is dedicated to prove the following

\begin{prop}\label{bidisc-prop}
 The horosphere boundary and the Gromov boundary of the bidisc, endowed with their respective topologies, are not homeomorphic.
\end{prop}

We first start with a description of the horopheres of sequences converging to some boundary point in $\partial \mathbb D^2$. In order to avoid burdening notations, in this section, we simply write $E(\{u_n\},R)$ to denote the horospheres of an admissible sequence $\{u_n\}$ in $\D^2$ with respect to the base point $(0,0)$. 

\begin{lem}\label{pdisc-lem} Every sequence $\{u_n=(u_n^1,u_n^2)\}\subset \D^2$ converging to a point  $(p_1,p_2)\in\partial \D^2$ is admissible. Moreover, 
\begin{itemize}
\item[(i)] if  $p=(e^{it},p_2) \in \partial \mathbb D^2$ with $t\in \R$ and $|p_2|<1$ ({\sl respectively} $p=(p_1,e^{it}) \in \partial \mathbb D^2$ with $|p_1|<1$, $t\in \R$))  for every $R>0$ it holds
 \[
 E(\{u_n\},R) =  E_{\mathbb D}(e^{it},R) \times \mathbb D
 \]
  ({\sl respect.},  $E(\{u_n\},R) = \mathbb D \times E_{\mathbb D}(e^{it},R)$).
\item[(ii)] If $p=(e^{it_1},e^{it_2}) \in \partial \mathbb D^2$, $t_1,t_2\in \R$, let $T_1:=\limsup_n\left(\frac{1-|u_n^1|^2}{1-|u_n^2|^2}\right)$ and $T_2:=\limsup_n\left(\frac{1-|u_n^2|^2}{1-|u_n^1|^2}\right)$. Then, for every $R>0$:
\begin{itemize}
 \item[(a)] $(T_1 > 1, \ T_2 > 1) \ \Rightarrow E(\{u_n\},R) = E_\mathbb D(e^{it_1},R) \times E_\mathbb D(e^{it_2},R)$,
 \item[(b)] $(T_1 > 1, \ T_2 \leq 1) \ \Rightarrow E(\{u_n\},R) = E_\mathbb D(e^{it_1},R/T_2) \times E_\mathbb D(e^{it_2},R)$,
 \item[(c)] $(T_1 \leq 1, \ T_2 > 1) \ \Rightarrow E(\{u_n\},R) = E_\mathbb D(e^{it_1},R) \times E_\mathbb D(e^{it_2},R/T_1)$
\end{itemize}
with the convention $E_\mathbb D(e^{it_1},+\infty) = \mathbb D$.
\end{itemize}
\end{lem}

\begin{proof}
 (i). Let $p=(e^{it},p_2) \in \partial \mathbb D^2$ with $|p_2|<1$. Then for $R>0$, for $w=(w_1,w_2) \in \mathbb D$ and for sufficiently large $n$:
 $$
 K_{\mathbb D^2}(u_n,w)-K_{\mathbb D^2}(u_n,0) = K_{\mathbb D}(u_n^1,w_1)-K_{\mathbb D}(u_n^1,0).
 $$
 Hence $E_{\mathbb D^2}(\{u_n\},R)=E_{\mathbb D}(\{u_n^1\},R) \times \mathbb D$.  

 (ii). By definition
 $$
 K_{\mathbb D^2}(u_n,w) - K_{\mathbb D^2}(u_n,0) = \max\left(K_{\mathbb D}(u_n^1,w_1),K_{\mathbb D}(u_n^2,w_2)\right) - \max\left(K_{\mathbb D}(u_n^1,0),K_{\mathbb D}(u_n^2,0)\right).
 $$
 Following \cite[pp.264-266]{Aba}, we obtain
 \begin{equation}\label{horo-eq}
 \begin{split}
 & \limsup_{n\to \infty}(K_{\mathbb D^2}(u_n,w)-K_{\mathbb D^2}(u_n,0))\\ &= 
  \frac{1}{2}\log\left(\max\left(\frac{|w_1-e^{it_1}|^2}{1-|w_1|^2}\min(1,T_2),
 \frac{|w_2-e^{it_2}|^2}{1-|w_2|^2}\min(1,T_1)\right)\right).
\end{split}
\end{equation}
 Hence

$$
w \in E(\{u_n\},R) \Leftrightarrow \left\{
\begin{array}{lll}
 \displaystyle \frac{|w_1-e^{it_1}|^2}{1-|w_1|^2}\min(1,T_2) & < & R\\
  & & \\
  \displaystyle \frac{|w_2-e^{it_2}|^2}{1-|w_2|^2}\min(1,T_1) & < & R
\end{array}
\right.
$$

The conclusion follows directly from that equivalence.
\end{proof}

Now we show that, in fact, we can restrict ourselves to particular admissible sequences converging to a boundary point:

\begin{prop}\label{top-prop}
Let $\{u_n\}$ be an admissible sequence in $\D^2$. Then $\{u_n\}$ is equivalent to one and only one of the following sequences:
\begin{enumerate}
\item $\{w^{(1)}_n(p_1,p_2):=(p_1(1-\frac{1}{n}), p_2(1-\frac{1}{n}))\}_{n\in \N}$ for some $p_1,p_2\in \partial \D$, 
\item $\{w^{(2)}_n(p):=(p(1-\frac{1}{n}), 0)\}_{n\in \N}$ for some $p\in \partial \D$,
\item $\{w^{(3)}_n(p):=(0,  p(1-\frac{1}{n}))\}_{n\in \N}$ for some $p\in \partial \D$.
\end{enumerate}
In particular, every admissible sequence in $\D^2$ is equivalent to an admissible sequence in $\D^2$ which converges to a point in $\partial\D^2$.
\end{prop}

\begin{proof}
First of all, it is clear by Lemma \ref{pdisc-lem} that the  horospheres of $\{w_n^{(1)}(p_1,p_2)\}$ are $E_\D(p_1, R)\times E_\D(p_2, R)$, the horospheres of $\{w_n^{(2)}(p)\}$ are $E(p, R)\times \D$ and of $\{w_n^{(3)}(p)\}$ are $\D \times E(p, R)$, for all $R>0$. Hence the sequences $\{w_n^{(1)}(p_1,p_2)\}$, $\{w_n^{(2)}(p)\}$ and $\{w_n^{(3)}(p)\}$ are not equivalent.

 Then, as in \eqref{int-hor}, we have
\begin{equation}\label{int-hor2}
E(\{u_n\}, R)=\bigcap_{\{v_n\}\in \mathcal M} E(\{v_n\}, R),
\end{equation}
where $\mathcal M$ denotes the set of all subsequences of $\{u_n\}$ converging to a boundary point. In particular, every sequence $\{v_n\}\subset \mathcal M$ is admissible (by \eqref{int-hor2} or by Lemma \ref{pdisc-lem}).  Note that \eqref{int-hor2} implies that if $\{v_n\}, \{\tilde{v}_n\}\in \mathcal M$ then for every $R>0$ it holds
\begin{equation}\label{inter-pol}
E(\{v_n\}, R)\cap E(\{\tilde{v}_n\}, R)\neq \emptyset.
\end{equation}

The horospheres of converging sequences are described in Lemma \ref{pdisc-lem}. There are different cases to be considered:

{\sl Case 1.} Assume that there exist $p_1, p_2\in \partial \D$, $\{v_n\}, \{\tilde{v}_n\}\in \mathcal M$,  (possibly $\{v_n\}=\{\tilde{v}_n\}$) and $T\in (0,1]$ such that  
\[
E(\{v_n\}, R)\cap E(\{\tilde{v}_n\}, R)=E_\D(p_1, R)\times E_\D(p_2,R/T)
\]
 for every $R>0$. 

We claim that $\{u_n\}$ is equivalent to the sequence $\{w_n^{(1)}(p_1,p_2)\}$. Indeed, taking into account \eqref{inter-pol} and the list in Lemma~\ref{pdisc-lem},  we see that if $\{z_n\}\in \mathcal M$, then either  $E(\{z_n\}, R)=E_{\mathbb D}(p_1,R) \times E_{\mathbb D}(p_2,R/Q)$  or $E(\{z_n\}, R)=E_{\mathbb D}(p_1,R/Q) \times E_{\mathbb D}(p_2,R)$ for some $Q\in [0,1]$ (with the usual convention that, if $Q=0$ then $E_{\mathbb D}(p_1,R/Q)=\D$). In any case, $E(\{w^{(1)}_n(p_1,p_2)\}, R)\subset E(\{z_n\}, R)$. Therefore from \eqref{int-hor2}, 
\begin{equation*}
E_\D(p_1, R)\times E_\D(p_2, R)\subset E(\{u_n\}, R)\subset E_{\mathbb D}(p_1,R) \times E_{\mathbb D}(p_2,R/T). 
\end{equation*}
In particular, it is easy to see that for every $R>0$ there exists $R'>0$ such that $E(\{w_n^{(1)}(p_1,p_2)\}, R)\subset E(\{u_n\}, R)\subset E(\{w_n^{(1)}(p_1,p_2)\}, R')$. Hence $\{u_n\}$ is equivalent to $\{w^{(1)}_n(p_1,p_2)\}$.

{\sl Case 2.} Assume that there exist $p_1, p_2\in \partial \D$, $\{v_n\}, \{\tilde{v}_n\}\in \mathcal M$,  (possibly $\{v_n\}=\{\tilde{v}_n\}$) and $T\in (0,1]$ such that  
\[
E(\{v_n\}, R)\cap E(\{\tilde{v}_n\}, R)=E_\D(p_1, R/T)\times E_\D(p_2,R)
\]
 for every $R>0$. 

In this case the argument goes exactly as in Case 1,  and $\{u_n\}$ is equivalent to $\{w^{(1)}_n(p_1,p_2)\}$.
 
{\sl Case 3}. There are no sequences $\{v_n\}, \{\tilde{v}_n\}$ in $\mathcal M$ as in Case 1 or Case 2. Hence, by Lemma \ref{pdisc-lem},   for all  $\{v_n\}\in \mathcal M$ there exists $p_{\{v_n\}}\in \partial \D$ such that either $E(\{v_n\}, R)=E_\D(p_{\{v_n\}}, R)\times \D$, or $E(\{v_n\}, R)=\D\times E_\D(p_{\{v_n\}}, R)$  for all $R>0$.  Moreover, since we are excluding Case 1 and Case 2, if for some $\{v_n\}\in \mathcal M$ it holds $E(\{v_n\}, R)=E_\D(p_{\{v_n\}}, R)\times \D$, then for all $\{\tilde{v}_n\}\in \mathcal M$ it holds $E(\{\tilde{v}_n\}, R)=E_\D(p_{\{\tilde{v}_n\}}, R)\times \D$.  By \eqref{inter-pol},  there exists $p\in \partial \D$ such that $p_{\{v_n\}}=p$ for all $\{v_n\}\in \mathcal M$, hence $E(\{u_n\}, R)=E(p,R)\times \D$, and $\{u_n\}$ is equivalent to $\{u_n\}$ is equivalent to $\{w^{(2)}_n(p)\}$. 

Similarly, if $E(\{v_n\}, R)=\D\times E_\D(p_{\{v_n\}}, R)$, it follows that $\{u_n\}$ is equivalent to $\{w^{(3)}_n(p)\}$. 
\end{proof}

Now we are ready to prove the main result of this section:

\begin{prop}\label{nonhaus}
The horosphere topology of $\partial_H\mathbb D^2$ induced from  $\widehat{\mathbb D^2}$ is trivial.
\end{prop}

\begin{proof}
We are going to show that  the only non empty closed subset of $\partial_H\mathbb D^2$ is  $\partial_H\mathbb D^2$.  In order to do that, we prove that the closure in $\partial_H\D^2$ of the class of any admissible sequence is $\partial_H\mathbb D^2$.

Since by Proposition \ref{top-prop} every admissible sequence is equivalent to  $\{w^{(1)}_n(p_1,p_2)\}$, $\{w^{(2)}_n(p)\}$ or $\{w^{(3)}_n(p)\}$, we only need to consider those sequences. 

In fact,  the following observations yield immediately the result:
\begin{itemize}
\item[A)] for every $p_1, p_2\in \partial \D$, the closure of the point $[\{w^{(1)}_n(p_1,p_2)\}]\in \partial_H\D^2$ contains the points $[\{w_n^{(2)}(p_1)\}]$ and $[\{w_n^{(3)}(p_2)\}]$. 
\item [B)] for every $p\in \partial \D$,  the closure of the point $[\{w^{(2)}_n(p)\}]\in \partial_H\D^2$ contains the points $[\{w^{(1)}_n(p,q_2)\}]$ for all $q_2\in \partial \D$.
\item[C)] for every $p\in \partial \D$,  the closure of the point $[\{w^{(3)}_n(p)\}]\in \partial_H\D^2$ contains the points $[\{w^{(1)}_n(q_1,p)\}]$ for all $q_1\in \partial \D$.
\end{itemize}
Indeed, it follows from A), B) and C) that the closure of the point $[\{w^{(1)}_n(p_1,p_2)\}]\in \partial_H\D^2$ also contains the points $[\{w^{(1)}_n(q_1,q_2)\}]$ for all $q_1, q_2 \in \partial \D$, the closure of the point $[\{w^{(2)}_n(p_1)\}]\in \partial_H\D^2$ also contains the points $[\{w^{(2)}_n(q_1)\}]$ for all $q_1 \in \partial \D$, and the closure of the point $[\{w^{(3)}_n(p_2)\}]\in \partial_H\D^2$ also contains the points $[\{w^{(3)}_n(q_2)\}]$ for all $q_2 \in \partial \D$.

We only show A), the proofs of B) and C) being similar.  By Lemma \ref{pdisc-lem}, it follows that for all $R>0$, 
\begin{equation*}
E(\{w_n^{(1)}(p_1,p_2)\}, R)\cap E(\{w_n^{(2)}(p_1)\}, R)\neq \emptyset.
\end{equation*}
 Hence, if we define the constant sequence $\{\underline{x_m}\}$ of elements in $\partial_H \mathbb D^2$  by $\underline{x_m} := [\{w_n^{(1)}(p_1,p_2)\}]$ for every $m\in \N$, then according to Definition~\ref{conv-out-out} the sequence $\{\underline{x_m}\}$ converges to $[\{w_n^{(2)}(p_1)\}]\in \partial_H\D^2$ when $m$ tends to infinity. Hence $[\{w_n^{(2)}(p_1)\}]$ belongs to the closure of $[\{w^{(1)}_n(p_1,p_2)\}]$. Similarly, $[\{w_n^{(3)}(p_2)\}]$ belongs to the closure of $[\{w^{(1)}_n(p_1,p_2)\}]$.
\end{proof}

 As an application of Theorem~\ref{extension}, Theorem~\ref{main-pseudo} and Proposition~\ref{nonhaus} we obtain a different proof of the well-known result due to H. Poincar\'e

\begin{prop}\label{nonbihol-prop}
There is no biholomorphism between the unit ball and the bidisc in $\mathbb C^2$.
\end{prop} 

\begin{proof}
 A biholomorphism between the unit ball and the bidisc would extends as a homeomorphism between their horosphere boundaries. This is not possible by Propositions~\ref{main-pseudo} and \ref{nonhaus}.
\end{proof}

In fact, one can use the previous arguments to show that the restriction of the horosphere topology of the polidisc $\D^n\subset \C^N$ to $\partial_H\D^n$ for every $n>1$ is not Hausdorff. Hence, by Theorem~\ref{extension} and Theorem~\ref{main-pseudo} one can see also that there exist no biholomorphisms between any bounded strongly pseudoconvex domain and the polydisc in $\C^N$, $n>1$.
 
In order to compare the Gromov boundary with the horosphere boundary of $\D^2$, we will prove now the following
\begin{prop}\label{nonhaus-gro}
The topology of $\partial_G^r \mathbb D^2$ is not trivial.
\end{prop}

\begin{proof} We choose the base point $(0,0)\in \D^2$.
Let $a:(-\infty,+\infty)\to \D$ be the geodesic for the hyperbolic distance of $\D$ whose image is the segment $(-1,1)$ and satisfying $\lim_{t \to - \infty}a(t) = -1,\ \lim_{t \to + \infty}a(t) = 1$.

Let $(1,0)_G$ be the point of $\partial_G^r \mathbb D^2$ represented by the geodesic ray $\alpha(t):=(a(t),0)$, $t\geq 0$, and let $(-1,0)_G$ be the point of $\partial_G^r \mathbb D^2$ represented by the geodesic ray $\beta(s):=(a(-s),0)$, $s\geq 0$.

Since $K_{\D^2}((z_1,z_2),(w_1,w_2))=\max_{j=1,2} \{K_\D(z_j,w_j)\}$, it follows that every geodesic rays $\gamma:[0,+\infty)\to \D^2$ such that $\gamma(0)=(0,0)$ is of the form $\gamma(t)=(f(t), g(t))$, where either $f$ or $g$ (or both) is a geodesic in $\D$. Hence, if $\gamma(t)=(f_1(t),f_2(t))$ is a geodesic ray in $\D^2$ equivalent to $(1,0)_G$, it follows that $f_1=a$ and $f_2((0,+\infty))$ is relatively compact in $\D$. Thus, every geodesic ray representing $(1,0)_G$ is of the form $(a(t), f(t))$ where  $\sup_{t\in [0,+\infty)}|f(t)|<1$ and $K_\D(0, f(t))\leq t$ for every $t \geq 0$. 
Similarly, if $\gamma$ is a geodesic ray equivalent to $(-1,0)_G$ it follows that $\gamma(s)=(a(-s), g(s))$, with  $\sup_{t\in [0,+\infty)}|g(t)|<1$ and $K_\D(0, g(t)) \leq t$ for every $t \geq 0$.

Now, let  $\gamma^+(t):=(a(t), f(t))$ be a geodesic ray representing $(1,0)_G$ and $\gamma^-(s):=(a(-s), g(s))$ be a geodesic ray representing $(-1,0)_G$. The Gromov product between $\gamma^+$ and $\gamma^-$ with respect to $(0,0)$ is:
\begin{equation*}
\begin{split}
2(\gamma^+(t), \gamma^-(s))_{(0,0)}&=2((a(t),f(t)), (a(-s), g(s))_{(0,0)}\\&=K_{\D^2}((a(t),f(t)), (0,0))+K_{\D^2}((a(-s),g(s)), (0,0))\\&-K_{\D^2}((a(t),f(t)), (a(-s),g(s)))\\&=t+s-\max\{K_\D(a(t),a(-s)), K_\D(f(t),g(s))\}
\\&=t+s-\max\{s+t, K_\D(f(t),g(s))\}\leq 0.
\end{split}
\end{equation*}
Hence, the Gromov product is always $0$.

Therefore, for any  couple of geodesic rays $\gamma^+$ representing $(1,0)_G$ and $\gamma^-$ representing $(-1,0)_G$ it holds
\[
\liminf_{t,s\to+\infty}(\gamma^+(t), \gamma^-(s))_{(0,0)}= 0.
\]
This implies that for every $r>0$ the point $(-1,0)_G$ does not belong to the open set $V((1,0)_G, r)$. Hence, for every $r>0$, the set $\partial_G^r \mathbb D^2\setminus V((1,0)_G,r)$ is a non empty closed set which does not contain $(1,0)_G$. In particular, the Gromov topology on $\partial_G^r \mathbb D^2$ is not trivial.
\end{proof}

Proposition~\ref{bidisc-prop} follows at once from Proposition \ref{nonhaus} and Proposition \ref{nonhaus-gro}.

\section{Boundary behavior}\label{App}

In this section we apply the results developed so far to study boundary behavior of univalent mappings defined on strongly pseudoconvex smooth domains. If $F:D \to \C^N$ is a holomorphic map and $p\in \partial D$, we denote by $\Gamma(F;p)$ the cluster set of $F$ at $p$, that is,
\[
\Gamma(F;p):=\{q\in \C^N: \exists \{w_n\}\subset D, F(w_n)\to q\}.
\]

We start by the following result:

\begin{theo}\label{convergenceI}
Let $D\subset \C^N$ be a  bounded strongly pseudoconvex domain with $C^3$ boundary. Let $F:D \to \Omega$ be a biholomorphism. Let $p\in \partial D$. Then there exists $\underline{x}\in \partial_H \Omega$ such that
\[
\Gamma(F;p)=\hbox{I}^H_\Omega(\underline{x}).
\] 
In particular, if there exists $q\in \partial \Omega$  such that $\hbox{I}^H_\Omega(\underline{x})=\{q\}$, then $\lim_{z\to p}F(z)=q$.
\end{theo}

\begin{proof}
By Theorem \ref{extension}, $F$ defines a homeomorphism $\hat{F}: \hat{D}\to \hat{\Omega}$. By Theorem \ref{main-pseudo}, there is a homeomorphism $\Theta: \hat{D}\to \overline{D}$ such that $\Theta(z)=z$ for $z\in D$. Therefore, a sequence $\{z_n\}\subset D$ converges to $p$ if and only if $\{\Theta^{-1}(z_n)\}$ converges to $\Theta^{-1}(p)\in \partial_H D$. Let $\underline{x}:=\hat{F}(\Theta^{-1}(p))$. Hence, the sequence $\{F(z_n)=\hat{F}(\Theta^{-1}(z_n))\}$ converges to $\underline{x}$ if and only if $\{z_n\}$ converges to $p$. From this the result follows.
\end{proof}

A first application is the following well known result (see, \cite{Fef}):

\begin{cor}\label{homeo-strongly}
Let $D\subset \C^N$ be a  bounded strongly pseudoconvex domain with $C^3$ boundary. Let $F:D \to \Omega$ be a biholomorphism. If $\Omega$ is a bounded strongly pseudoconvex domain with $C^3$ boundary then $F$ extends to a homeomorphism from $\overline{D}$ to $\overline{\Omega}$. 
\end{cor}

\begin{proof}
It is enough to apply Theorem \ref{convergenceI} and Proposition \ref{strongly-pseudo-boundary} to both $F$ and $F^{-1}$. 
\end{proof}

Another application of our theory gives a positive answer to Conjecture 3.1.(a) in \cite{MS}, in fact, proving not only continuous extension, but extension as homeomorphism:

\begin{cor}\label{continuo-ball}
Let $D\subset \C^N$ be a bounded strongly convex domain with $C^3$ boundary. Let $F:D \to \Omega$ be a biholomorphism. If $\Omega$ is a bounded convex domain,  then $F$ extends as a homeomorphism  from $\overline{D}$ to $\overline{\Omega}$. 
\end{cor}

\begin{proof}
By  Theorem \ref{ball-biholo-convex} and Theorem \ref{convergenceI}, for every $p\in \partial D$ the limit $F(p):=\lim_{z\to p}F(z)$ exists. In order to see that the map $F:\overline{D}\to \overline{\Omega}$ is continuous, we have only to show that if $\{p_j\}\subset \partial D$ is a sequence converging to $p\in \partial D$, then $F(p_j)\to F(p)$.  Indeed, by Proposition \ref{strongly-pseudo-boundary}, there exist $\underline{x}_j, \underline{x}\in \partial_H D$ such that $\hbox{I}^H_{D}(\underline{x}_j)=\{p_j\}$, $\hbox{I}^H_{D}(\underline{x}_j)=\{p\}$ and $\{\underline{x}_j\}$ converges to $\underline{x}$ in the horosphere topology of $D$. Hence $\{\hat{F}(\underline{x}_j)\}$ converges to $\{\hat{F}(\underline{x})\}$ in the horosphere topology of $\Omega$. Moreover, $\hbox{II}^H_{\Omega}(\hat{F}(\underline{x}_j))=\hbox{I}^H_\Omega(\hat{F}(\underline{x}_j))=\{F(p_j)\}$ and similarly $\hbox{II}^H_{\Omega}(\hat{F}(\underline{x}))=\hbox{I}^H_{\Omega}(\hat{F}(\underline{x}))=\{
F(p)\}$. Hence, $F(p_j)\to F(p)$ by 
Corollary \ref{converge-bene}.

Therefore, $F:\overline{D}\to \overline{\Omega}$ is continuous. Since $\overline{D}$ is compact, in order to prove that $F$ is a homeomorphism, we only need to prove that it is injective. 

We argue by contradiction. Assume $p_0, p_1\in \partial D$ and $q:=F(p_0)=F(p_1)\in\partial \Omega$. Let $\{u_n\}\subset D$ be a sequence converging to $p_0$ and $\{v_n\}\subset D$ be a sequence converging to $p_1$. Then $\{u_n\}$ and  $\{v_n\}$ are admissible by  Proposition \ref{pointstrict} and not equivalent  by Proposition \ref{convergenceadmsp}. Let $x\in D$ and $R>0$  be such that $V:=E^{D}_x(\{u_n\}, R)\cap E^{D}_x(\{v_n\}, R)\neq \emptyset$. Note that $V$ is open and relatively compact in $D$ because $\overline{E^{D}_x(\{u_n\}, R)}\cap \partial D=\{p_0\}$ and $\overline{E^{D}_x(\{v_n\}, R)}\cap \partial D=\{p_1\}$. 

Since $F$ is a biholomorphism it maps horospheres onto horospheres. Moreover, $F(E^{D}_x(\{u_n\}, R))=E^\Omega_{F(x)}(\{F(u_n)\}, R)$ and $F(E^{D}_x(\{v_n\}, R))=E^\Omega_{F(x)}(\{F(v_n)\}, R)$ are convex by Proposition \ref{horo-convex}. Hence, 
\[
F(V)=F(E^{D}_x(\{u_n\}, R)\cap E^{D}_x(\{v_n\}, R))=E^{\Omega}_{F(x)}(\{F(u_n)\}, R)\cap E^{\Omega}_{F(x)}(\{F(v_n)\}, R)
\]
is open, convex and relatively compact in $\Omega$. 

Now, $q\in \overline{E^{\Omega}_{F(x)}(\{F(u_n)\}, R)}\cap \overline{E^{\Omega}_{F(x)}(\{F(v_n)\}, R)}$. Let $z_0\in F(V)$. Since the two horospheres are both open, convex, and $z_0$ is contained in both horospheres, the real segment $\gamma:=\{tz_0+(1-t)q: t\in (0,1]\}$ as well is contained in  both horospheres. That is, $\gamma\subset F(V)$. But then $F(V)$ is not relatively compact in $\Omega$, a contradiction.  
\end{proof}

With a similar argument, using Proposition \ref{strict-biholo-convex} instead of Theorem \ref{ball-biholo-convex}, we have

\begin{cor}\label{continuo-strict}
Let $D\subset \C^N$ be a  bounded strongly pseudoconvex domain with $C^3$ boundary. Let $F:D \to \Omega$ be a biholomorphism. If $\Omega$ is a strictly $\C$-linearly bounded convex domain,   then $F$ extends as a homeomorphism  from $\overline{D}$ to $\overline{\Omega}$.
\end{cor}

Now we consider non-tangential limits. In fact, in our theory, the right notion to consider is that of {\sl $E$-limits}. Let $D\subset \C^N$ be a bounded strongly pseudoconvex domain with $C^3$ boundary. If $p\in \partial D$, by Proposition \ref{strongly-pseudo-boundary}, there exists $\underline{x}_p\in \partial_H D$ such that $\hbox{I}^H_{D}(\underline{x}_j)=\{p\}$. Given a map $f:\D \to \C^N$, we denote by $\Gamma_E(f;p)$ the cluster set of $f$ at $p$ along sequences $E$-converging to $\underline{x}_p$ (see Definition \ref{E-lim}), namely,
\[
\Gamma_E(f;p)=\{q\in \C^N: \exists \{z_n\}\subset D : E-\lim_{n\to \infty} z_n=\underline{x}_p, f(z_n)\to q\}.  
\] 

A slight modification of the proof of Theorem \ref{convergenceI}, taking into account that $F$ maps sequences E-converging to $\underline{y}$ to sequences E-converging to $\hat{F}(\underline{y})$ and Lemma \ref{II-to-hor}, gives the following:

\begin{theo}\label{convergenceII}
Let $D\subset \C^N$ be a  bounded strongly pseudoconvex domain with $C^3$ boundary. Let $F:D \to \Omega$ be a biholomorphism, $x\in \Omega$. Let $p\in \partial D$. Then there exists $\underline{x}\in \partial_H \Omega$ such that
\[
\Gamma_E(f;p)=\hbox{II}^H_\Omega(\underline{x})=\bigcap_{R>0} \overline{E_x(\{u_n\}, R)}^{\mathbb C \mathbb P^N},
\]
where $\{u_n\}\subset D$ is any admissible sequence representing $\underline{x}$. 
\end{theo}

Finally, let $D\subset \C^N$ be a bounded strongly pseudoconvex domain with $C^3$ boundary and $p\in \partial D$. If $f:D \to \C^N$ is a map, we denote by $\Gamma_{NT}(f;p)$ the cluster set of $f$ along sequences converging to $p$ non-tangentially.  By  Proposition \ref{pointstrict}, 
\[
\Gamma_{NT}(f;p)\subseteq \Gamma_E(f;p)\subseteq \Gamma(f;p).
\]

In particular, by Theorem \ref{convergenceII}ÃÂ  we have:

\begin{cor}\label{strong-ext}
Let $D\subset \C^N$ be a  bounded strongly pseudoconvex domain with $C^3$ boundary. Let $F:D \to \Omega$ be a biholomorphism and assume that for every $\underline{x}\in \partial_H \Omega$ the horosphere principal part $\hbox{II}^H_\Omega(\underline{x})$ consists of one point. Then for every $p\in \partial D$ the non-tangential limit $\angle\lim_{z\to p}F(z)$ exists.
\end{cor}

As a spin off result of our work, we prove the following Wolff-Denjoy theorem, which gives a (partial) positive answer to a conjecture in \cite{AR} (see \cite[Rmk. 3.3]{AR}). Let $D\subset \C^N$ be a bounded domain. Let $f:D\to D$ be holomorphic. A point $q\in \overline{D}$ belongs to the {\sl target set} $T(f)$ of $f$ if there exist a sequence $\{k_m\}\subset \N$ converging to $\infty$ and $z\in D$ such that $\lim_{m\to \infty}f^{k_m}(z)=q$. 

\begin{prop}\label{Denjoy}
Let $D\subset \C^N$ be a bounded convex domain. Assume that  either $D$ is  biholomorphic to a strongly convex domain with $C^3$ boundary or $D$ is $\C$-strictly linearly convex and biholomorphic to a bounded strongly pseudoconvex domain with $C^3$ boundary. Let $f: D \to D$ be holomorphic without fixed points in $D$. Then there exists  exists $p\in \partial D$ such that $T(f)=\{p\}$.
\end{prop}

\begin{proof} 
Since $f$ has no fixed point in $D$ then $\{f^k\}$ is compactly divergent (see \cite{ABZ, A2}). By \cite[Lemma 3.10]{AR} there exists a Busemann admissible sequence $\{u_n\}$ such that for every $R>0$, 
\begin{equation}\label{dw-lemma}
f(E_x(\{u_n\}, R))\subset E_x(\{u_n\}, R).
\end{equation}
Hence the result follows from either Theorem \ref{ball-biholo-convex} or Proposition \ref{strict-biholo-convex}. 
\end{proof}

\begin{rem}
As we already pointed out, there exist convex domains biholomorphic to the unit ball which are not $\C$-strictly linearly convex, but  for which the Denjoy-Wolff theorem holds by Proposition \ref{Denjoy}. We conjecture that in fact the result in Proposition \ref{Denjoy} holds for every bounded convex domain $D\subset\C^N$  whose boundary does not contain non-constant analytic discs.
\end{rem}

\end{document}